\documentclass[11pt]{article}

\usepackage[T1]{fontenc} 
\usepackage[utf8]{inputenc} 
\usepackage{amsmath, amssymb, amsthm} \usepackage{empheq}
\usepackage{graphicx} 
\usepackage{xcolor} 
\usepackage{tcolorbox} \definecolor{myblue}{RGB}{56, 103, 214}
\definecolor{mygreen}{RGB}{7,153,146}
\usepackage[colorlinks=true,linkcolor=myblue,citecolor=mygreen,urlcolor=blue]{hyperref} 
\usepackage{bm} 
\usepackage{array, booktabs, multirow} 
\usepackage{float} 
\usepackage{mathrsfs} 
\usepackage{appendix} 
\usepackage[a4paper]{geometry} \usepackage{caption}

\usepackage[numbers, sort&compress]{natbib}

\numberwithin{equation}{section}

\tcbset{colback=blue!5!white, colframe=blue!60!black, boxrule=0.6pt, arc=2mm}

\theoremstyle{plain} \newtheorem{theorem}{Theorem}[section]
\newtheorem{lemma}[theorem]{Lemma}

\theoremstyle{remark} \newtheorem{remark}[theorem]{Remark}

\theoremstyle{definition} 
\newtheorem{example}{Example}[section]

\newcommand{\bn}{\bm{n}} \newcommand{\bx}{\bm{x}} \newcommand{\bH}{\bm{H}}
\newcommand{\ri}{\mathrm{i}} 
 \newcommand{\cH}{\mathcal{H}}
\newcommand{\cL}{\mathcal{L}} \newcommand{\bCurl}{\text{curl}}

\newcommand{\nn}{\nonumber}

\usepackage{mathtools}

\newcommand{\Ih}{\mathcal{I}_h} \newcommand{\Reop}{\operatorname{Re}}
\newcommand{\Imop}{\operatorname{Im}} \newcommand{\dd}{\,\mathrm{d}}
\newcommand{\bA}{\bm{A}} \newcommand{\bE}{\bm{E}} \newcommand{\bv}{\bm{v}}
\newcommand{\Grad}{\nabla} \newcommand{\Curl}{\nabla \times} \newcommand{\Lam}{\Lambda}
\newcommand{\bJ}{\bm{J}} \newcommand{\bu}{\bm{u}}
\newcommand{\sphi}{\mathcal{S}_{\phi}}

\title{A Gauge-Invariant Finite Element Method for the Time-Dependent
  Ginzburg--Landau Equations}

\author{ Huadong~Gao, \thanks{ School of Mathematics and Statistics, Huazhong
    University of Science and Technology, Wuhan 430074, P.R. China ({\tt
      huadong@hust.edu.cn}). The work of the author was supported in part by the
    National Key Research and Development Program of China (No.~2023YFC3804500)
    and the National Natural Science Foundation of China (No.~12231003).} \quad
  Yuhui~Huang \thanks{Corresponding author. School of Mathematics and Statistics, Huazhong University
    of Science and Technology, Wuhan 430074, P.R. China. ({\tt
      yuhui\_huang@hust.edu.cn}) } ~~ and ~~ Wen~Xie \thanks{ School of
    Mathematics and Statistics, Changsha University of Science and Technology,
    Changsha 410114, Hunan, P.R. China ({\tt wenxiemath@163.com}) } } \date{}

\begin{document}

\maketitle

\begin{abstract}
  It is well known that the Ginzburg--Landau equations are gauge invariant. The
  most popular gauge choice in practical simulations is the temporal gauge, also
  known as the zero electric potential gauge. However, the commonly used
  N\'ed\'elec approximations may produce spurious giant vortices on domains with
  re-entrant corners. To address this issue while retaining the N\'ed\'elec
  approximation, we propose a finite element scheme that preserves the gauge
  invariance of the time-dependent Ginzburg--Landau equations at the discrete
  level. Under the temporal gauge, the method employs lowest-order N\'ed\'elec
  edge elements for the magnetic potential and continuous piecewise-linear
  elements for the order parameter on triangular and tetrahedral meshes, and
  yields a linear, decoupled formulation that requires only the solution of
  linear systems at each time step. We prove that the scheme is gauge invariant,
  preserves a discrete maximum bound principle for the order parameter,
  $\|\psi_h^n\|_{L^\infty(\Omega)}\le 1$, and satisfies a discrete energy law under mild step-size and
  mesh conditions. Numerical experiments confirm these theoretical properties.
  In particular, on domains with re-entrant corners, the scheme computes the
  correct physical solution, whereas conventional temporal-gauge discretizations
  with N\'ed\'elec elements produce an unphysical solution. \vskip 0.2in
  
\end{abstract}

\section{Introduction}
\label{sec:introduction}

The time-dependent Ginzburg--Landau (TDGL) equations, first proposed by
Schmid~\cite{Schmid1966} and later derived from the microscopic BCS
theory~\cite{Bardeen1957} by Gor'kov and Eliashberg~\cite{Gorkov1968}, provide a
widely used mesoscopic model for the evolution of superconducting states,
governing the complex order parameter $\psi$, the magnetic potential $\bA$, and
the electric potential $\phi$. In their nondimensional form, the equations read
\begin{empheq}[left=\empheqlbrace]{align}
  &\eta\partial_t\psi+\ri\eta\kappa\phi\psi +\left(\dfrac{\ri}{\kappa}\nabla+\bA\right)^2\psi +\bigl(|\psi|^2-1\bigr)\psi
  =0,\label{eq:tdgl-no-fixed:psi}\\
  &\partial_t\bA+\Curl\!\bigl(\Curl\bA\bigr) +\Grad \phi +\Reop\!\left[\psi^*\left(\dfrac{\ri}{\kappa}\nabla+\bA\right)\psi\right] =\Curl H\,,\label{eq:tdgl-no-fixed:A}
\end{empheq}
subject to the boundary and initial conditions
\begin{align}
  \left(\dfrac{\ri}{\kappa}\nabla+\bA\right)\psi\cdot\bn=0,
  \qquad
  &(\Curl\bA) \times \bn=H \times \bn,
    \qquad
    \bigl(\partial_t\bA+\nabla\phi\bigr)\cdot\bn=0
    \quad\text{on }\partial\Omega,
    \label{eq:continuous-boundary-conditions}\\
  &\psi(\cdot,0)=\psi_0,
    \qquad
    \bA(\cdot,0)=\bA_0\,,
    \label{eq:continuous-initial-conditions}
\end{align}
where $\Omega \subset \mathbb{R}^d$ ($d=2,3$) is a bounded Lipschitz domain with outward
unit normal vector $\mathbf{n}$ on $\partial\Omega$. The positive parameters $\eta$ and $\kappa$
represent the normalized conductivity and the Ginzburg--Landau parameter,
respectively. The applied magnetic field $H$ is taken to be a real constant for
$d=2$, and a constant real vector field for $d=3$. In \eqref{eq:tdgl-no-fixed:psi}--\eqref{eq:tdgl-no-fixed:A}, $|\psi|^2$
represents the Cooper pair density; specifically, $|\psi|^2=1$ and $|\psi|^2=0$
signify the perfectly superconducting state and the normal state, respectively,
while $0<|\psi|^2<1$ corresponds to the mixed (vortex) state.

When $d=2$, for any vector-valued function $\bm{u}=[u_1,u_2]^T$ and scalar
function $f$, we define the two-dimensional curl operators as
\begin{equation}
  \Curl\bm{u}:=\frac{\partial u_2}{\partial x_1} - \frac{\partial u_1}{\partial x_2},
  \qquad
  \Curl f := \bigl[ \frac{\partial f}{\partial x_2}, - \frac{\partial f}{\partial x_1} \bigr]^T.
  \label{eq:continuous-curl-conventions}
\end{equation}

The TDGL system is gauge invariant \cite{Du1992a}: if $(\psi, \bA, \phi)$ is a solution, then for
any sufficiently smooth scalar function $\chi$ the transformed variables
\begin{equation}
  \widetilde{\psi} = \psi e^{i\kappa\chi}, \qquad \widetilde{\bA}
  = \bA + \nabla\chi, \qquad \widetilde{\phi} = \phi - \partial_t\chi
\end{equation}
also satisfy the same system. Under this gauge transformation, all physically
observable quantities remain invariant, including the Cooper pair density
$|\widetilde{\psi}|^2 = |\psi|^2$ and the magnetic induction
$\Curl\widetilde{\bA} = \Curl\bA$.

Without gauge fixing, the solution $(\psi, \bA, \phi)$ to the TDGL system is
non-unique. A widely used choice is the temporal gauge (zero electric
potential), i.e., $\phi \equiv 0$. Under this gauge, the system \eqref{eq:tdgl-no-fixed:psi}--\eqref{eq:continuous-initial-conditions} reduces to
\begin{empheq}[left=\empheqlbrace]{align}
  &\eta\partial_t\psi +\left(\dfrac{\ri}{\kappa}\nabla+\bA\right)^2\psi +\bigl(|\psi|^2-1\bigr)\psi = 0, \label{eq:tdgl-zero-phi:psi}\\
  &\partial_t\bA+\Curl\!\bigl(\Curl\bA\bigr) +\Reop\!\left[\psi^*\left(\dfrac{\ri}{\kappa}\nabla+\bA\right)\psi\right] =\Curl H, \label{eq:tdgl-zero-phi:A}
\end{empheq}
subject to the boundary and initial conditions
\begin{align}
  \left(\dfrac{\ri}{\kappa}\nabla+\bA\right)\psi\cdot\bn=0,
  \qquad
  &(\Curl\bA) \times \bn=H \times \bn,
    \qquad
    \bA\cdot\bn=0
    \quad\text{on }\partial\Omega,
    \label{eq:continuous-boundary-conditions-zero-phi}\\
  &\psi(\cdot,0)=\psi_0,
    \qquad
    \bA(\cdot,0)=\bA_0\,.
    \label{eq:continuous-initial-conditions-zero-phi}
\end{align}
The above system can be viewed as a gradient flow of the following
Ginzburg--Landau energy:
\begin{align}
  \mathcal{G}(\psi, \bA) :=
  \frac{1}{2} \int_{\Omega} \left| \left(\frac{\mathrm{i}}{\kappa}\nabla + \bA \right) \psi \right|^2 \mathrm{d}x
  + \frac{1}{4} \int_{\Omega} \left( |\psi|^2 - 1 \right)^2 \mathrm{d}x
  + \frac{1}{2} \|\Curl \bA - H \|_{L^2(\Omega)}^2 \,.
\end{align}

Reviews of the analysis and numerical methods
for the Ginzburg--Landau model can be found in \cite{Du1992a,Du1994a,Du2005a}.
Numerical methods for the TDGL equations have been developed
for the two most popular gauge choices: the Lorentz gauge and the temporal
gauge. Under the Lorentz gauge ($\phi=-\mathrm{div} \bA$), conventional Lagrange
finite element methods (FEMs) were studied in \cite{Chen1995a,Gao2014a}. The magnetic potential of the
TDGL system on a nonconvex polygon is not in $H^1(\Omega)$ in general, so a finite
element method that computes $\bA$ directly in nodal (Lagrange) spaces may
converge to an incorrect solution. To circumvent this difficulty, a mixed method
was proposed in \cite{Gao2015a} and Hodge-decomposition-based FEMs were introduced in \cite{Li2015a,Li2016}.
Under the temporal gauge ($\phi=0$), the TDGL equations \eqref{eq:tdgl-zero-phi:psi}--\eqref{eq:tdgl-zero-phi:A} take a simpler
form than those under the Lorentz gauge. Moreover, the energy decay property of
the model can be readily derived by taking inner products of the system with
$(\partial_t\psi,\partial_t\bA)$. Accordingly, numerical methods for the TDGL equations under the
temporal gauge have been extensively studied; see, e.g., \cite{Duan2022,Yang2014a,Wu2018a,Hong2023a,Ma2023a}. To address the
singularity of the magnetic potential $\bA$, N\'ed\'elec elements are used to
discretize equation \eqref{eq:tdgl-zero-phi:A} in \cite{Duan2022,Hong2023a,Ma2023a}. These studies also presented numerical
experiments on an L-shaped domain to demonstrate the convergence of their
methods. It is widely believed that using N\'ed\'elec elements for $\bA$ allows
the vortex pattern to be computed correctly.

However, in a recent study, Tain, Caputo, and Danaila~\cite{Tain2024a} conducted systematic
tests to assess the influence of the gauge choice on the numerical accuracy of
solutions to the TDGL equations computed using the mixed FEM proposed in \cite{Gao2015a}.
They introduced a general gauge with $\phi= - \omega \mathrm{div} \bA$, where $\omega=0$
corresponds to the temporal gauge, $\omega=1$ to the Lorentz gauge, and $\omega=\infty$ to the
Coulomb gauge. They showed in \cite{Tain2024a} that if $\omega$ is close to zero, i.e., as the
gauge approaches the temporal gauge, the mixed FEM in \cite{Gao2015a} may yield spurious
giant vortices near re-entrant corners. Indeed, our numerical test in Section~\ref{sec:numerical-experiments}
also shows that using N\'ed\'elec edge elements to discretize equation (\ref{eq:tdgl-zero-phi:A}) for
$\bA$ directly yields a very similar spurious vortex pattern. These observations suggest that this failure occurs in conventional temporal-gauge discretizations.

Nevertheless, the temporal gauge remains the most attractive framework for
numerical simulation, owing to its simple form and natural energy decay. This
naturally leads to the following problem: Can the TDGL equations be solved
reliably in the temporal gauge? To address this problem, we develop a
gauge-invariant finite element method whose discrete solutions are independent
of the gauge. This enables us to compute in the temporal gauge for convenience
while avoiding the artifacts of conventional temporal-gauge discretizations.

Several gauge-invariant numerical methods have been proposed for the TDGL and
related models. Du~\cite{Du1998b} studied gauge-invariant finite-difference approximations
of the TDGL model on rectangular grids. Du, Nicolaides, and Wu~\cite{Du1998c} also proposed
and analyzed a covolume approximation of the Ginzburg--Landau model on
triangular meshes that preserves discrete gauge invariance. Later, Gao, Ju,
and Xie \cite{Gao2019a} analyzed a gauge-invariant stabilized semi-implicit Euler finite
difference method. In the physics literature, this finite difference method is
known as the ``{\it linked variable approach}'', e.g., \cite{Coskun1997,Frahm1991,Gropp1996a,Gunter2002,Kato1993,Sadovskyy2015a,Winiecki2002}. In contrast, the
systematic construction of gauge-invariant finite element methods on
unstructured triangular and tetrahedral meshes was pioneered by Christiansen and
Halvorsen~\cite{Christiansen2011d} for the magnetic Schr\"odinger eigenvalue problem. This approach
was later extended in \cite{Christiansen2012a} to a general gauge theory, establishing a discrete
framework that preserves the essential gauge structure of the continuous problem
on such grids. To the best of our knowledge, no gauge-invariant finite element
method on triangular or tetrahedral meshes has been developed for the TDGL
equations. In this paper, we extend the spatial gauge-invariant framework of \cite{Christiansen2011d,Christiansen2012a}
to the TDGL equations by constructing a fully discrete FEM that preserves gauge
invariance. The resulting method is decoupled and linear at each time step and
uses the lowest-order finite element spaces on triangular and tetrahedral
meshes. While the scheme is gauge-invariant in general, the present computations
are restricted to the temporal gauge. We rigorously prove that the scheme
preserves a discrete maximum bound principle for the order parameter,
$|\psi_h^n|\le 1$, and a discrete energy law under mild restrictions on the time
step and the mesh. In domains with re-entrant corners, the scheme computes the
correct physical solution, whereas conventional temporal-gauge discretizations
with N\'ed\'elec edge elements produce an unphysical solution on the same mesh.

The rest of the paper is organized as follows. In Section~\ref{sec:fully-discrete-scheme}, we introduce the
finite element spaces and present the fully discrete gauge-invariant scheme,
along with its formulation in the temporal gauge. In Section~\ref{sec:theoretical-analysis}, we establish
the discrete maximum bound principle and discrete energy stability. In
Section~\ref{sec:numerical-experiments}, we present several numerical examples to demonstrate the accuracy
and stability of the method, including computations on nonconvex domains.
Finally, concluding remarks are given in Section~\ref{sec:conclusion}.

\section{A gauge-invariant finite element scheme}
\label{sec:fully-discrete-scheme}

In this section, we introduce the notation and finite element setting, and
present a fully discrete gauge-invariant scheme for the TDGL equations \eqref{eq:tdgl-no-fixed:psi}--\eqref{eq:tdgl-no-fixed:A}
with the boundary and initial conditions \eqref{eq:continuous-boundary-conditions}--\eqref{eq:continuous-initial-conditions}.

\subsection{Notation and Finite Element Spaces}
\label{subsec:discrete-spaces}

Let $W^{k,p}(\Omega)$ denote the standard Sobolev space on $\Omega$, and let
$W^{k,p}_0(\Omega)$ be its subspace with vanishing trace \cite{Adams2003a}. As is standard, we
write $H^{k}(\Omega):=W^{k,2}(\Omega)$ and $H^{k}_0(\Omega):=W^{k,2}_0(\Omega)$, with
$H^{0}(\Omega)=L^2(\Omega)$. For complex-valued functions, $\cH^{k}(\Omega)$ denotes the space
of functions whose real and imaginary parts belong to $H^{k}(\Omega)$, and we set
$\cL^2(\Omega):=\cH^{0}(\Omega)$. For any $u,v\in\cL^2(\Omega)$, the corresponding inner product
is given by
\[
  (u,v):=\int_\Omega u\,v^{*}\,\dd \bx.
\]
Moreover, we define the curl-conforming space by
\begin{equation}
  \bH(\bCurl,\Omega):=
  \bigl\{ \bv \in [L^2(\Omega)]^d \mid \Curl\bv \in \mathbf{L}^2(\Omega) \bigr\},
\end{equation}
where $\mathbf{L}^2(\Omega) = L^2(\Omega)$ for $d=2$ and $\mathbf{L}^2(\Omega) = [L^2(\Omega)]^3$
for $d=3$ \cite{Monk2003a}. The associated norm is
\[
  \|\bv\|_{\bH(\bCurl)} = \|\bv\|_{\bH(\bCurl,\Omega)} := \bigl(\|\bv\|^2_{L^2(\Omega)}+\|\Curl\bv\|^2_{\mathbf{L}^2(\Omega)}\bigr)^{\frac12}.
\]
For simplicity, we omit $\Omega$ and write $\bH(\bCurl)$ when no ambiguity arises.

Let $\mathcal T_h$ be a conforming shape-regular tetrahedral partition (or triangular partition in 2D) of $\Omega$
with mesh size $h:=\max_{T\in \mathcal{T}_h}\operatorname{diam}(T)$.
Let $\mathcal{N}_{h}= \{ \bm{x}_i\}$ be the set of nodes of $\mathcal{T}_h$.
We assign a fixed orientation to each edge and denote the resulting set of oriented edges by
$\mathcal{E}_{h}$, where each geometric edge appears exactly once. The oriented
edge pointing from vertex $\bx_i$ to vertex $\bx_j$ is denoted by
$e_{i,j}=(i,j)\in\mathcal E_h$, with unit tangent vector
$\boldsymbol t_e:=(\bx_j-\bx_i)/|\bx_j-\bx_i|$. Throughout this paper, $C$
denotes a generic positive constant independent of the mesh size $h$.

We introduce the finite element spaces
\[
  V_h=\mathcal P_1(\mathcal T_h;\mathbb C), \qquad X_h=\mathcal{ND}_1(\mathcal T_h), \qquad S_h=\mathcal P_1(\mathcal T_h;\mathbb R),
\]
where $V_h$ and $S_h$ consist of continuous piecewise linear complex-valued and
real-valued functions, respectively, and $X_h$ denotes the lowest-order
N\'ed\'elec edge element space of the first kind \cite{Monk2003a}. These spaces satisfy the
subcomplex inclusion property
\[
  \Grad S_h\subset X_h.
\]
In addition, let $\mathcal{I}_h$ denote the nodal Lagrange interpolation
operator and $\Pi_h$ denote the N\'ed\'elec projection operator.

Let $\{\varphi_i\}$ denote the standard scalar nodal basis functions. For any
$u_h\in V_h$, we write
\[u_i:=u_h(\bx_i).\] The lumped $\cL^2$ inner product is defined by
\begin{equation}
  (u,v)_{0,h} := \int_{\Omega} \mathcal{I}_h(u v^*) \dd \bx = \sum_{\bm{x}_i \in \mathcal{N}_h} m_i (\mathcal{I}_h u)_i (\mathcal{I}_h v)^{*}_i,
  \qquad m_i := \int_\Omega\varphi_i\,\dd \bx = \frac{1}{d+1}\sum_{T\ni\bx_i}|T|,\label{eq:def-of-lumping-inner}
\end{equation}
which is obtained via row-sum lumping of the mass matrix.

To construct a gauge-invariant discretization of the covariant derivative,
following the gradient mass-lumping construction of Christiansen and
Halvorsen~\cite{Christiansen2011d}, we associate with each edge $e=(i,j)\in\mathcal{E}_h$ the weight
\begin{equation}
  \mu_e := -\int_\Omega \Grad\varphi_i\cdot\Grad\varphi_j\,\dd \bx
  = \sum_{T\supset e}\mu_{e,\scriptscriptstyle T}, \qquad \mu_{e,\scriptscriptstyle T} := -\int_T \Grad\varphi_i\cdot\Grad\varphi_j\,\dd \bx.
\end{equation}
The local contribution $\mu_{e,\scriptscriptstyle T}$ admits an explicit geometric representation
depending on the dimension:
\begin{itemize}
  \item For $d=2$: let $T\supset e$ be a triangle with opposite vertex $k$, and let
        $\theta_{\scriptscriptstyle T,e}$ be the interior angle of $T$ opposite to $e$. Then
        \begin{equation}
          \mu_{e,\scriptscriptstyle T} = \frac12\cot\theta_{\scriptscriptstyle T, e}, \qquad \mu_e
          = \frac12\sum_{T\supset e}\cot\theta_{\scriptscriptstyle T, e}.
        \end{equation}
  \item For $d=3$: let $T\supset e$ be a tetrahedron with $e=e_{i,j}$, and let
        $e' = (k,l)$ denote the opposite edge in $T$ joining the remaining two
        vertices $\bx_k$ and $\bx_l$. If $\theta_{\scriptscriptstyle T,e'}$ denotes the interior
        dihedral angle of $T$ along the opposite edge $e'$, then
        \begin{equation}
          \mu_{e,\scriptscriptstyle T} = \frac16\,|\bx_k-\bx_l|\cot\theta_{\scriptscriptstyle T,e'}, \qquad \mu_e
          = \frac16\sum_{T\supset e}|\bx_k-\bx_l|\cot\theta_{\scriptscriptstyle T,e'}.
        \end{equation}
\end{itemize}
In both cases, the weight $\mu_e$ is independent of the chosen orientation of $e$.

We further assume that the mesh family $\{\mathcal{T}_h\}_{h>0}$  satisfies the nonnegative
edge-weight condition
\begin{equation}
  \int_\Omega \Grad\varphi_i\cdot\Grad\varphi_j\,\dd \bx \le 0 \qquad \text{for every } e=(i,j)\in\mathcal{E}_h,\label{eq:mesh-degree-assume}
\end{equation}
which guarantees that $\mu_e \ge 0$ for all edges. This condition is well known in
connection with discrete maximum principles and is readily satisfied by
non-obtuse meshes; we refer the reader to Christiansen and Halvorsen~\cite{Christiansen2011d} for
further details.

For any $\bv_h \in X_h$, let
\[v_e := \int_e \bv_h \cdot \boldsymbol{t}_e\,\dd s
\]
denote its degree of freedom on the edge $e \in \mathcal{E}_h$. Induced by the
edge weights $\mu_e$, we define the edge-weighted bilinear form on $X_h$ and its
associated seminorm by
\begin{equation}
  \langle \bu_h, \bv_h \rangle_{\mathcal{E}_h} := \sum_{e\in\mathcal{E}_h} \mu_e u_e v_e, \qquad |\bv_h|_{\mathcal{E}_{h}}^2 := \langle \bv_h, \bv_h \rangle_{\mathcal{E}_h}, \qquad \forall\,\bu_h, \bv_h \in X_h.
  \label{eq:edge-lumped-inner-product}
\end{equation}

The following lemma bounds this seminorm uniformly in terms of the standard
$L^2$ norm on $X_h$.
\begin{lemma}[Uniform upper bound for the edge-weighted seminorm on $X_h$]
  \label{lem:edge-seminorm-bound}
  Let $\{\mathcal T_h\}_{h>0}$ be a shape-regular family of simplicial meshes in
  $\mathbb R^d$ ($d=2,3$). Then there exists a positive constant $C$,
  independent of $h$ and depending only on the shape-regularity of the mesh,
  such that
  \begin{equation}
    |\bv_h|_{\mathcal{E}_{h}}  \le C \|\bv_h\|_{L^2(\Omega)}, \qquad \forall\,\bv_h\in X_h\,.\label{eq:upper-bound-edge-norms}
  \end{equation}
  We therefore define
  \begin{equation}
    \Lambda := \sup_{0\neq\bv_h\in X_h} \frac{ |\bv_h|_{\mathcal{E}_{h}} }{\|\bv_h\|_{L^2(\Omega)}}.\label{eq:def-of-lambda}
  \end{equation}
\end{lemma}
\begin{proof}
  For each simplex $T\in\mathcal T_h$, let $\widehat T$ be the reference simplex
  with affine mapping $F_T:\widehat T\to T$, and let
  $\widehat{\bv}_h \in \mathcal{ND}_1(\widehat T)$ be the pull-back of $\bv_h|_T$
  under the covariant Piola transformation (see Monk~\cite{Monk2003a}). Then the edge degrees
  of freedom satisfy
  \begin{equation}
    \int_e \bv_h\cdot\boldsymbol t_e\,\dd s = \int_{\widehat e}\widehat{\bv}_h \cdot\widehat{\boldsymbol t}_{\widehat e}\,\dd\widehat{s}, \qquad \forall\,e \in \mathcal{E}_{h}.
  \end{equation}
  Since the reference N\'ed\'elec space \(\mathcal{ND}_1(\widehat T)\) is
  finite-dimensional, by norm equivalence on the reference element, there exists
  a constant \(C_1>0\) such that
  \begin{equation}
    \sum_{\widehat e\subset\partial\widehat T}
    \left|
      \int_{\widehat e}\widehat{\bv}_h\cdot\widehat{\boldsymbol t}_{\widehat e}\,\dd\widehat{s}
    \right|^2
    \le
    C_1\|\widehat{\bv}_h\|_{L^2(\widehat T)}^2.
  \end{equation}
  The scaling properties of the covariant Piola mapping and mesh
  shape-regularity yield the inverse scaling estimate
  \begin{equation}
    \|\widehat{\bv}_h\|_{L^2(\widehat T)}^2
    \le
    C_2 h_T^{2-d}\|\bv_h\|_{L^2(T)}^2.
  \end{equation}
  Meanwhile, shape-regularity ensures that $|\mu_{e,\scriptscriptstyle T}|\le C_3 h_T^{d-2}$. This
  immediately leads to the local estimate
  \begin{equation}
    \sum_{e\subset\partial T}\mu_{e,\scriptscriptstyle T}|v_e|^2
    \le
    \sum_{e\subset\partial T}|\mu_{e,\scriptscriptstyle T}|\,|v_e|^2
    \le
    C_4\|\bv_h\|_{L^2(T)}^2,
  \end{equation}
  where $C_4=C_{1}C_{2}C_{3}$. Finally, using \(\mu_e=\sum_{T\supset e}\mu_{e,\scriptscriptstyle T}\), we arrive
  at
  \begin{equation*}
    \sum_{e\in\mathcal E_h}\mu_e|v_e|^2
    =
    \sum_{T\in\mathcal T_h}\sum_{e\subset\partial T}\mu_{e,\scriptscriptstyle T}|v_e|^2
    \le
    C_4\sum_{T\in\mathcal T_h} \|\bv_h\|_{L^2(T)}^2
    \le
    C_4\|\bv_h\|_{L^2(\Omega)}^2,
  \end{equation*}
  which completes the proof.
\end{proof}

\subsection{Fully Discrete Gauge-Invariant Scheme}
\label{subsec:fully-discrete-scheme}

Let $N \in \mathbb{N}$ and consider a uniform partition of the time interval
$[0, T]$ with step size $\tau := T/N$ and discrete time levels $t_n := n\tau$ for
$n = 0, 1, \dots, N$. At time level $t_n$, let $A_e^n$ denote the N\'ed\'elec
degree of freedom of $\bA_h^n \in X_h$ associated with each oriented edge
$e = (i,j) \in \mathcal{E}_h$.

To preserve gauge invariance at the discrete level, following~\cite{Christiansen2011d,Christiansen2012a}, we define the
link variable $U_e^n$ and the discrete covariant derivative operator
$D^n:V_h\to X_h+\ri X_h$, which acts as the discrete counterpart of the continuous
covariant derivative $(\frac{\ri}{\kappa}\nabla+\bA)\psi$, by
\begin{equation}
  U_e^n:=\exp(\ri\kappa A_e^n), \qquad
  (D^n\psi_h)_e:=\frac{\ri}{\kappa}
  \left(\psi_j-U_e^n\psi_i\right),
  \qquad \forall\,\psi_h\in V_h.
\end{equation}
Here $X_h+\ri X_h =\{\mathbf u_h+\ri\mathbf v_h \mid \mathbf u_h,\mathbf v_h\in X_h\}$,
$\psi_i:=\psi_h(\bx_i)$ is the nodal value at vertex $\bx_i$, and $(D^n\psi_h)_e$ denotes
the complex-valued N\'ed\'elec edge degree of freedom of $D^n\psi_h$ on the
oriented edge $e=(i,j)\in\mathcal E_h$. For notational brevity, we write
$D_e^n \psi_h := (D^n \psi_h)_e$. We define the discrete covariant energy form
\begin{equation}
  a_h^n(\psi_h, \varphi_h) := \sum_{e\in\mathcal E_h} \mu_e D_e^n \psi_h\, (D_e^n \varphi_h)^{*},
  \qquad \forall\,\psi_h, \varphi_h \in V_h,\label{eq:def-of-covariance-der}
\end{equation}

We introduce the discrete electric potential $\phi_h^{n+1/2}\in S_h$ with nodal
values denoted by $\phi_i^{n+1/2}:=\phi_h^{n+1/2}(\bx_i)$. We define
$\sphi \psi_h^{\,n+1}\in V_h$ by its nodal values
\begin{equation}
  \bigl(\sphi \psi_{h}^{n+1}\bigr)_i := \exp\!\left(\ri\kappa\tau\phi_i^{n+1/2}\right) \psi_i^{n+1}.
\end{equation}
The discrete covariant time derivative and discrete electric field are then
defined by
\begin{equation}
  \mathcal{D}_{\phi}^{\tau} \psi_h^n := \frac{\sphi \psi_h^{\,n+1}-\psi_h^n}{\tau} \in V_h,
  \qquad
  \bE_h^{n+1/2} := \frac{\bA_h^{n+1}-\bA_h^n}{\tau} +\nabla\phi_h^{n+1/2} \in X_h.\label{eq:def-time-diff-phi}
\end{equation}
The discrete supercurrent $\bJ_h^{n+1/2} \in X_h$ is defined by its edge degrees
of freedom:
\begin{equation}
  J_e^{n+1/2} := \operatorname{Re}\!\left[ (\sphi \psi_{h}^{\,n+1})_{j}^* D_e^n (\sphi \psi_h^{\,n+1}) \right]
  =\frac{1}{\kappa} \Imop\!\left[ (\sphi \psi_{h}^{\,n+1})_{j}^* U_e^n (\sphi \psi_i^{\,n+1}) \right]\,.
  \label{eq:discrete-current-def}
\end{equation}
With the above notation, the fully discrete scheme is to find $(\psi_h^{n+1}, \bA_{h}^{n+1})\in V_h\times X_{h}$
such that for all $(\varphi_h, \bv_h) \in V_h \times X_h$
\begin{empheq}[left=\empheqlbrace]{align}
  & \eta \left(\mathcal{D}^{\tau}_{\phi} \psi_h^n,\varphi_h\right)_{0,h}
  + a_h^n\left(\sphi \psi_h^{\,n+1},\varphi_h\right)
  + \left((|\psi_h^n|^2-1)\psi_h^n,\varphi_h\right)_{0,h} = 0,
  \label{eq:local-order-parameter-equation}
  \\
  & \left(\bE_h^{n+1/2},\bv_h\right) + \left(\operatorname{curl}\bA_h^{n+1},\operatorname{curl}\bv_h\right)
  + \left\langle \bJ_h^{n+1/2},\bv_h\right\rangle_{\mathcal E_h} = \left(H,\operatorname{curl}\bv_h\right).
  \label{eq:local-vector-potential-equation}
\end{empheq}
The corresponding discrete energy at time level $t_n$ is defined by
\begin{align}
  \mathcal{G}_h^n(\psi_h^n, \bA_h^n)
  &:= \frac{1}{2}\sum_{e\in\mathcal E_h} \mu_e |D_e^n \psi_h^n|^{2}
  + \frac{1}{4}\sum_{\bm{x}_i \in \mathcal{N}_h} m_i \left(|\psi_i^n|^2 - 1\right)^2
  + \frac{1}{2}\|\Curl\bA_h^n - H\|_{L^2(\Omega)}^2
    \nn\\
  &= \frac12 a_h^n(\psi_h^n, \psi_h^n) + \frac{1}{4} (|\psi_h^n|^2-1,|\psi_h^n|^2-1)_{0,h}
    + \frac{1}{2}\|\Curl\bA_h^n - H\|_{L^2(\Omega)}^2 \,.
  \label{eq:local-discrete-energy}
\end{align}

To examine the gauge properties of the proposed scheme, we introduce the
discrete gauge transformations. Let $\chi_h^{n} \in S_h$ be a real-valued function.
The discrete gauge-transformed variables are defined by
\begin{equation}
  \widetilde\psi_h^{\,n} := \mathcal{I}_h\!\left(\exp(\ri\kappa\chi_h^n)\psi_h^n\right),
  \qquad
  \widetilde{\bA}_h^{\,n}:= \bA_h^n + \nabla\chi_h^n,
  \qquad
  \widetilde\phi_h^{\,n+1/2} := \phi_h^{n+1/2} - \frac{\chi_h^{n+1} - \chi_h^n}{\tau},
  \label{eq:local-discrete-gauge-transform}
\end{equation}
which satisfy
$\widetilde{\psi}_h^{\,n}(\bx_i) = \exp(\ri\kappa\chi_h^n(\bx_i))\psi_h^n(\bx_i)$ at each
vertex $\bx_i$.

The following lemma establishes that the proposed fully discrete FEM
\eqref{eq:local-order-parameter-equation}--\eqref{eq:local-vector-potential-equation} is invariant
under the discrete gauge transformation \eqref{eq:local-discrete-gauge-transform}.

\begin{lemma}[Discrete gauge invariance]
  \label{lem:local-discrete-gauge-covariance}
  The discrete system \eqref{eq:local-order-parameter-equation}--\eqref{eq:local-vector-potential-equation} is invariant under the discrete gauge
  transformation \eqref{eq:local-discrete-gauge-transform}. Consequently, the equalities
  \[
    |\widetilde\psi_i^{n}| = |\psi_i^n| \qquad \text{and} \qquad \Curl\widetilde{\bA}_h^n = \Curl\bA_h^n
  \]
  hold, and the discrete energy is invariant:
  \begin{equation}
    \mathcal{G}_h^n(\widetilde\psi_h^n,\widetilde{\bA}_h^n) = \mathcal{G}_h^n(\psi_h^n,\bA_h^n).\label{eq:energy-gauge-covariance}
  \end{equation}
\end{lemma}

\begin{proof}
  To simplify the notation, let \(g_i^n := \exp(\ri\kappa\chi_h^n(\bx_i))\) throughout
  the proof, then we have
  \begin{equation}
    \widetilde{\psi}^{n}_{i} = g_{i}^{n} \psi^{n}_{i}\,.\label{eq:local-psi-covariance}
  \end{equation}
  For each oriented edge $e=(i,j)$, the transformation of $\bA_h^n$ gives
  \[
    \widetilde A_e^{\,n} = A_e^n+\chi_j^n-\chi_i^n.
  \]
  Hence the spatial links satisfy
  \begin{equation}
    \widetilde U_e^n:= \exp(\ri\kappa \widetilde{A}_e^{\,n})
    =\exp(\ri\kappa(A_e^n+\chi_j^n-\chi_i^n))=
    g_j^nU_e^n(g_i^n)^*.
    \label{eq:local-link-covariance}
  \end{equation}

  Applying the definition of the transported order parameter to the transformed
  variables, we obtain at each node
  \begin{align}
    \left(\widetilde{\mathcal{S}}_{\phi} \widetilde{\psi}^{n+1}_{h}\right)_{i}&:=\exp(\ri\kappa\tau\widetilde{\phi}_{i}^{\,n+1/2})
                                                                          \widetilde{\psi}_i^{\,n+1} = \exp(\ri\kappa\tau\phi_{i}^{\,n+1/2})\exp[\ri \kappa (\chi^{n}_{i}-\chi^{n+1}_{i})]g_{i}^{n+1}
                                                                          \psi_i^{n+1}\nn
    \\
                                                                        &=
                                                                          \exp(\ri\kappa\tau\phi_i^{n+1/2})g_i^n
                                                                          \psi_i^{n+1}=
                                                                          g_i^n\left(\sphi \psi_{h}^{n+1}\right)_{i}\,.
                                                                          \label{eq:local-transported-covariance}
  \end{align}
  Consequently, by~\eqref{eq:def-time-diff-phi}, \eqref{eq:local-transported-covariance} and $|g_{i}^{n}|=1$ we have
  \begin{equation}
    \left(\widetilde{\mathcal{D}}_{\phi}^{\,\tau}
      \widetilde{\psi}_h^n\right)_{i}
    :=
    g_i^n
    \left(\mathcal{D}^{\tau}_{\phi}\psi_h^n\right)_{i},\qquad
    (|\widetilde\psi_i^n|^2-1)\widetilde\psi_i^n
    =
    g_i^n(|\psi_i^n|^2-1)\psi_i^n.
    \label{eq:local-time-and-nonlinear-covariance}
  \end{equation}
  Given any $\varphi_h\in V_h$, we define $\widetilde{\varphi}_{h} \in V_{h}$ as
  \(\widetilde{\varphi}_h = \Ih\!\left(\exp(\ri\kappa\chi_h^n)\varphi_h\right)\), i.e.,
  \begin{equation}
    \widetilde{\varphi}_i :=\widetilde{\varphi}_h(\bx_i) =\exp(\ri\kappa\chi_i^n)\varphi_i = g^n_i \varphi_i.\label{eq:local-varphi-covariance}
  \end{equation}
  By \eqref{eq:local-psi-covariance} and \eqref{eq:local-link-covariance}, we obtain
  \begin{equation}
    \widetilde{D}_e^n
    \left(\widetilde{\mathcal{S}}_{\phi} \widetilde{\psi}_h^{\,n+1}\right)
    =
    g_j^nD_e^n\left(\sphi \psi_h^{\,n+1}\right),\quad\text{and}\quad
    \widetilde D_e^n\widetilde{\varphi}_h
    = g_j^nD_e^n\varphi_h,\label{eq:local-space-covariance}
  \end{equation}
  where $\widetilde D_e^n$ is the covariant edge operator associated with
  $\widetilde U_e^n$. Then we have
  \begin{align}
    \widetilde{a}_{h}^{n} (\widetilde{\mathcal{S}}_{\phi} \widetilde{\psi}_{h}^{n+1},\widetilde{\varphi}_{h})&:=  \sum_{e\in\mathcal E_h} \mu_e \widetilde{D}_e^n (\widetilde{\mathcal{S}}_{\phi}\widetilde{\psi}_{h}^{\,n+1})\, (\widetilde{D}_e^n \widetilde{\varphi}_h)^{*}\nn\\
                                                                                                 &=\sum_{e\in\mathcal E_h}\mu_e D_e^n\left(\sphi \psi_h^{\,n+1}\right)\left(D_e^n\varphi_{h}\right)^{*}=a_{h}^{n} (\mathcal{S}_{\phi} \psi_{h}^{n+1},\varphi_{h})\,.\label{eq:local-energy-covariance}
  \end{align}
  The invariance of \eqref{eq:local-order-parameter-equation} then follows from \eqref{eq:local-time-and-nonlinear-covariance}, \eqref{eq:local-varphi-covariance}, \eqref{eq:local-energy-covariance}, and the fact that
  $\varphi_h\mapsto\widetilde\varphi_h$ is a bijection on $V_h$.

  For the vector-potential equation, direct substitution gives
  \begin{align}
    \widetilde{\bE}_h^{n+1/2} &:=
                                \frac{\widetilde{\bA}_h^{\,n+1}-\widetilde{\bA}_{h}^{\,n}}{\tau} +\nabla\widetilde{\phi}_h^{\,n+1/2}\nn\\
                              &=\frac{\bA_h^{n+1}-\bA_h^n}{\tau}
                                +\frac{\Grad\chi_h^{n+1}-\Grad\chi_h^n}{\tau}
                                +\Grad\left(
                                \phi_h^{n+1/2}
                                -\frac{\chi_h^{n+1}-\chi_h^n}{\tau}
                                \right)\nn\\
                              &=\bE_h^{n+1/2}\,.\label{eq:local-electric-covariance}
  \end{align}
  Moreover, $\Curl\Grad\chi_h^n=0$ implies
  \begin{equation}
    \Curl\widetilde{\bA}_h^n=\Curl\bA_h^n.\label{eq:local-curl-covariance}
  \end{equation}
  The edge current is unchanged as well. By \eqref{eq:local-transported-covariance},
  \eqref{eq:local-space-covariance} and $|g_j^n|=1$, there holds
  \begin{align}
    \widetilde J_e^{\,n+1/2}
    & :=\operatorname{Re}\!\left[ (\widetilde{\mathcal{S}}_{\phi} \widetilde{\psi}_{h}^{\,n+1})_{j}^* \widetilde{D}_e^n (\widetilde{\mathcal{S}}_{\phi}\widetilde{\psi}_h^{\,n+1}) \right]\nn\\
    &=\Reop\!\left[
      \left(\sphi \psi_h^{\,n+1}\right)_{j}^{*}(g^{n}_{j})^{*}
      g_{j}^{n}D_e^n(\sphi \psi_h^{\,n+1})
      \right]
      \nn \\
    & =J_e^{\,n+1/2}.\label{eq:local-supercurrent-covariance}
  \end{align}
  Combining \eqref{eq:local-electric-covariance}, \eqref{eq:local-curl-covariance}, and \eqref{eq:local-supercurrent-covariance} establishes the invariance of \eqref{eq:local-vector-potential-equation} for any test
  function $\bv_h\in X_h$. Finally,
  \begin{align}
    \mathcal{G}_{h}(\widetilde{\psi}^{n}_{h}, \widetilde{\bA}^{n}_{h})
    &:= \frac{1}{2}\sum_{e\in\mathcal E_h} \mu_e |\widetilde{D}_e^n \widetilde{\psi}^n_h|^{2}
      + \frac{1}{4}\sum_{\bm{x}_i \in \mathcal{N}_h} m_i \left(|\widetilde{\psi}_i^n|^2 - 1\right)^2
      + \frac{1}{2}\|\Curl\widetilde{\bA}_h^n - H\|_{L^2(\Omega)}^2
      \nn\\
    &= \frac{1}{2}\sum_{e\in\mathcal E_h} \mu_e |g_{j}^{n}|^{2} |D_e^n \psi^{n}_{h}|^{2}
      + \frac{1}{4}\sum_{\bm{x}_i \in \mathcal{N}_h} m_i \left(|g_{i}^{n}|^{2}|\psi_i^n|^2 - 1\right)^2
      + \frac{1}{2}\|\Curl\bA_h^n - H\|_{L^2(\Omega)}^2\nn\\
    &=\frac{1}{2}\sum_{e\in\mathcal E_h} \mu_e |D_e^n \psi^{n}_h|^{2}
      + \frac{1}{4}\sum_{\bm{x}_i \in \mathcal{N}_h} m_i \left(|\psi_i^n|^2 - 1\right)^2
      + \frac{1}{2}\|\Curl\bA_h^n - H\|_{L^2(\Omega)}^2\nn\\
    &=\mathcal{G}_{h} (\psi^{n}_{h},\bA^{n}_{h})\,.\nn
  \end{align}
  This establishes the desired gauge invariance.
\end{proof}

\subsection{Fully Discrete Scheme in the Temporal Gauge}

Owing to the discrete gauge invariance established above, a particular gauge can
be chosen without loss of generality. For practical computations, we adopt the
most popular temporal gauge
\begin{equation}
  \phi_h^{n+1/2} = 0, \qquad n \ge 0.
  \label{eq:discrete-zero-potential-gauge}
\end{equation}
In this gauge, the auxiliary field reduces to $\sphi \psi_h^{\,n+1} = \psi_h^{n+1}$,
and the edge supercurrent $\bJ_h^{n+1/2}$ becomes
\begin{equation}
  J_e^{n+1/2} = \frac{1}{\kappa} \Imop\!\left[ (\psi_j^{\,n+1})^* U_e^n \psi_i^{\,n+1} \right],
  \qquad \forall e = (i,j) \in \mathcal{E}_h.
\end{equation}
Starting from $(\psi_h^n, \bA_h^n) \in V_h \times X_h$, the scheme determines $\psi_h^{n+1}$
and $\bA_h^{n+1}$ sequentially by solving
\begin{empheq}[left=\empheqlbrace]{align}
  & \eta\Big( \frac{\psi_h^{n+1}-\psi_h^n}{\tau}, \varphi_h \Big)_{0,h} + a_h^n(\psi_h^{n+1},\varphi_h) + \left( (|\psi_h^n|^2-1)\psi_h^n, \varphi_h \right)_{0,h} = 0,
  \label{eq:discrete-tdgl-temporal-psi}\\
  &\Big( \frac{\bA_h^{n+1}-\bA_h^n}{\tau}, \bv_h \Big) + \left( \operatorname{curl}\bA_h^{n+1}, \operatorname{curl}\bv_h \right) + \langle \bJ_h^{n+1/2}, \bv_h \rangle_{\mathcal{E}_h} = \left( H, \operatorname{curl}\bv_h \right),
  \label{eq:discrete-tdgl-temporal-A}
\end{empheq}
for any test functions $\varphi_h\in V_h$ and $\bv_h\in X_h$. The initial conditions for
\eqref{eq:discrete-tdgl-temporal-psi}-\eqref{eq:discrete-tdgl-temporal-A} are given by projecting or interpolating the given continuous initial
data $(\psi_0, \bA_0)$ into the finite element spaces, such that
\begin{equation}
  \psi_h^0 = \mathcal{I}_h \psi_0 \in V_h, \qquad \bA_h^0 = \Pi_h \bA_0 \in X_h\,.
  \label{eq:discrete-initial-conditions}
\end{equation}

\begin{remark}[Well-posedness and decoupling]
  \label{rem:decoupled-solves}
  In the temporal gauge, the fully discrete scheme decouples into two successive
  linear subproblems. The existence and unique solvability of the solution
  $(\psi_h^{n+1}, \bA_h^{n+1})$ follow immediately at each time step, as the
  corresponding coefficient matrices are Hermitian (or symmetric)
  positive-definite.
\end{remark}

\section{Theoretical Analysis}
\label{sec:theoretical-analysis}

In this section, we establish two fundamental theoretical properties of the
proposed scheme: the discrete maximum bound principle (MBP) and energy
stability. For ease of presentation, we first carry out the proofs under the
temporal gauge.

\begin{theorem}[Discrete maximum bound principle]
  \label{thm:discrete-mbp}
  Let \(\{(\psi_h^n,\bA_h^n)\}_{n\ge0}\) be the solution of the discrete system
  \eqref{eq:discrete-tdgl-temporal-psi}--\eqref{eq:discrete-tdgl-temporal-A}. Under the nonnegative edge-weight condition~\eqref{eq:mesh-degree-assume} and the time-step restriction
  \(\tau\le\frac{\eta}{2}\),  if the initial data satisfies \(\|\psi_h^0\|_{L^\infty(\Omega)}\le1\), then
  \begin{equation}
    \|\psi_h^n\|_{L^\infty(\Omega)}\le1,
    \qquad n\ge0.
    \label{eq:mbp-infinity-bound}
  \end{equation}
\end{theorem}

\begin{proof}
  We proceed by induction. By assumption, $\|\psi_h^0\|_{L^\infty(\Omega)} \le 1$. Now, we
  assume inductively that $\|\psi_h^n\|_{L^\infty(\Omega)} \le 1$. We define the cutoff test
  function $\varphi_h \in V_h$ by its nodal values
  \begin{equation}
    \varphi_i =
    \begin{cases}
      \psi_i^{\,n+1}, & \text{if } |\psi_i^{n+1}| > 1, \\
      0,              & \text{if } |\psi_i^{n+1}| \le 1.
    \end{cases}
    \label{eq:mbp-test-function}
  \end{equation}
  Clearly, if $\|\psi_h^{n+1}\|_{L^\infty(\Omega)} \le 1$, then $\varphi_h \equiv 0$. To derive a
  contradiction, suppose that $\varphi_h \not\equiv 0$, which means that there exists at least
  one node $i$ such that $|\psi_i^{\,n+1}| > 1$. Substituting $\varphi_h$ into the
  discrete order-parameter equation \eqref{eq:discrete-tdgl-temporal-psi} yields
  \begin{equation}
    \eta \Big(\frac{\psi_h^{\,n+1} - \psi_h^n}{\tau}, \varphi_h\Big)_{0,h}
    + a_h^n(\psi_h^{\,n+1}, \varphi_h)
    + \bigl((|\psi_h^n|^2 - 1)\psi_h^n, \varphi_h\bigr)_{0,h} = 0.
    \label{eq:mbp-weak-form-temporal}
  \end{equation}

  First of all, we estimate the first and third terms on the left-hand side of
  \eqref{eq:mbp-weak-form-temporal}. Taking their real parts and by \eqref{eq:def-of-lumping-inner}, we obtain
  \begin{align}
    &\Reop\!\Big[
      \eta\Big(\frac{\psi_h^{\,n+1}-\psi_h^n}{\tau},\varphi_h\Big)_{0,h}
      +\bigl((|\psi_h^n|^2-1)\psi_h^n,\varphi_h\bigr)_{0,h} \Big]
      \nn\\
    &\quad=\sum_{|\psi_i^{\,n+1}|>1} m_i
      \left\{
      \frac{\eta}{\tau}\left[|\psi_i^{\,n+1}|^2 - \Reop\!\left(\psi_i^n(\psi_i^{\,n+1})^*\right)\right]
      +(|\psi_i^n|^2-1)\Reop\!\left(\psi_i^n(\psi_i^{\,n+1})^*\right)
      \right\}
      \nn \\
    &\quad = \sum_{|\psi_i^{\,n+1}|>1} m_i \left \{
      \frac{\eta}{\tau}|\psi_i^{\,n+1}|^2
      -\left(\frac{\eta}{\tau}+1-|\psi_i^n|^2\right)\Reop\!\left(\psi_i^n(\psi_i^{\,n+1})^*\right) \right \}.
      \label{eq:mbp-mass-reaction}
  \end{align}
  The bound $|\psi_i^n|\le 1$ yields
  \begin{equation}
    \frac{\eta}{\tau}+1-|\psi_i^n|^2 > 0.
    \label{eq:mbp-positive-coefficient}
  \end{equation}
  Since $\Reop(\psi_i^n(\psi_i^{\,n+1})^*) \le |\psi_i^n|\,|\psi_i^{\,n+1}|$, it follows
  that
  \begin{align}
    &\frac{\eta}{\tau}|\psi_i^{\,n+1}|^2
      -\left(\frac{\eta}{\tau}+1-|\psi_i^n|^2\right)\Reop\!\left(\psi_i^n(\psi_i^{\,n+1})^*\right)\nn \\
    &\quad\ge |\psi_i^{\,n+1}| \left[ \frac{\eta}{\tau}|\psi_i^{\,n+1}| -\left(\frac{\eta}{\tau}+1-|\psi_i^n|^2\right)|\psi_i^n| \right].\label{eq:mbp-mass-reaction-bound}
  \end{align}
  Furthermore, the condition $\tau \le \eta/2$ implies $\eta/\tau \ge 2$. Combined with
  $0\le|\psi_i^n|\le1$, we observe that
  \begin{equation}
    \frac{\eta}{\tau}-\left(\frac{\eta}{\tau}+1-|\psi_i^n|^2\right)|\psi_i^n|
    =(1-|\psi_i^n|)\left[\frac{\eta}{\tau}-|\psi_i^n|(1+|\psi_i^n|)\right] \ge 0.
    \label{eq:mbp-reaction-inequality}
  \end{equation}
  When $|\psi_i^{\,n+1}|>1$, the term inside the right-hand side of \eqref{eq:mbp-mass-reaction-bound} satisfies
  the strict inequality
  \begin{equation}
    \frac{\eta}{\tau}|\psi_i^{\,n+1}| - \left(\frac{\eta}{\tau}+1-|\psi_i^n|^2\right)|\psi_i^n|
    = \frac{\eta}{\tau}\left(|\psi_i^{\,n+1}|-1\right) + \frac{\eta}{\tau} - \left(\frac{\eta}{\tau}+1-|\psi_i^n|^2\right)|\psi_i^n| > 0.
    \label{eq:mbp-strict-reaction-inequality}
  \end{equation}
  Since $\varphi_h \not\equiv 0$, the sum in \eqref{eq:mbp-mass-reaction} is strictly positive,
  which leads to 
  \begin{equation}
    \Reop\!\Big[
      \eta\Big(\frac{\psi_h^{\,n+1}-\psi_h^n}{\tau},\varphi_h\Big)_{0,h}
      +\bigl((|\psi_h^n|^2-1)\psi_h^n,\varphi_h\bigr)_{0,h}
      \Big] > 0.\label{eq:mbp-temporal-reaction-strictly-positive}
  \end{equation}

  Next, we examine the second term on the left-hand side of \eqref{eq:mbp-weak-form-temporal}.
  Recalling \eqref{eq:def-of-covariance-der}, we have
  \begin{equation*}
    \Reop\left[a^{n}_{h}(\psi_{h}^{n+1},\varphi_{h})\right]=  \sum_{e\in\mathcal E_h} \mu_e \Reop\left[D_e^n \psi^{n+1}_h\, (D_e^n \varphi_h)^{*}\right]\,.
  \end{equation*}
  For any oriented edge $e = (i, j) \in \mathcal{E}_h$, we consider the following four cases:
  \begin{itemize}
    \item If both endpoints satisfy $|\psi_i^{\,n+1}|>1$ and $|\psi_j^{\,n+1}|>1$, then
          $D_e^n\varphi_h = D_e^n\psi_h^{\,n+1}$, and
          \begin{equation*}
            \Reop\!\left[D_e^n\psi_h^{\,n+1}\bigl(D_e^n\varphi_h\bigr)^*\right] = |D_e^n\psi_h^{\,n+1}|^2 \ge 0.
          \end{equation*}

    \item If only vertex $i$ satisfies $|\psi_i^{\,n+1}|>1$, then
          $D_e^n\varphi_h = -(\ri/\kappa)U_e^n\psi_i^{\,n+1}$. Recalling that \(|U_e^n| = 1\),
          a direct calculation yields
          \begin{align*}
            \Reop\!\left[D_e^n\psi_h^{\,n+1}\bigl(D_e^n\varphi_h\bigr)^*\right]\nn
            &= \frac{1}{\kappa^2}\left[|U^{n}_{e}|^{2}|\psi_i^{\,n+1}|^2 - \Reop\!\left(\psi_j^{\,n+1}(U_e^n)^*(\psi_i^{\,n+1})^*\right)\right] \\
            &\ge \frac{1}{\kappa^2}|\psi_i^{\,n+1}|\left(|\psi_i^{\,n+1}| - |\psi_j^{\,n+1}|\right) \ge 0\,,
          \end{align*}

    \item If only vertex $j$ satisfies $|\psi_j^{\,n+1}|>1$, then
          $D_e^n\varphi_h = (\ri/\kappa)\psi_j^{\,n+1}$, and similarly
          \begin{align*}
            \Reop\!\left[D_e^n\psi_h^{\,n+1}\bigl(D_e^n\varphi_h\bigr)^*\right]
            &= \frac{1}{\kappa^2}\left[|\psi_j^{\,n+1}|^2 - \Reop\!\left(U_e^n\psi_i^{\,n+1}(\psi_j^{\,n+1})^*\right)\right] \nn\\
            &\ge \frac{1}{\kappa^2}|\psi_j^{\,n+1}|\left(|\psi_j^{\,n+1}| - |\psi_i^{\,n+1}|\right) \ge 0.
          \end{align*}

    \item If $|\psi_i^{\,n+1}|\le1$ and $|\psi_j^{\,n+1}|\le1$, then $D_e^n\varphi_h = 0$, so
          the term vanishes identically.
  \end{itemize}
  Summing over all edges and taking into account the geometric weights
  $\mu_e \ge 0$ guaranteed by \eqref{eq:mesh-degree-assume}, we conclude that
  \begin{equation}
    \Reop\bigl[a_h^n(\psi_h^{\,n+1}, \varphi_h)\bigr] \ge 0.
    \label{eq:mbp-covariant-nonnegative}
  \end{equation}

  Combining \eqref{eq:mbp-temporal-reaction-strictly-positive} and \eqref{eq:mbp-covariant-nonnegative}, the real part of the left-hand side of \eqref{eq:mbp-weak-form-temporal} is strictly
  positive, yielding a contradiction. Therefore, the assumption $\varphi_h \not\equiv 0$
  cannot hold, implying
  \begin{equation}
    |\psi_i^{\,n+1}| \le 1 \qquad \text{for all } i.
    \label{eq:mbp-nodal-bound}
  \end{equation}
  Since $\psi_h^{n+1}$ is piecewise linear, the nodal bound \eqref{eq:mbp-nodal-bound} immediately implies
  that
  \begin{equation*}
    \|\psi_h^{n+1}\|_{L^\infty(\Omega)} \le 1.
  \end{equation*}
  By induction, this completes the proof.
\end{proof}

With the $L^\infty$-bound $\|\psi_h^n\|_{L^\infty(\Omega)} \le 1$ established in Theorem~\ref{thm:discrete-mbp},
we are now ready to state the discrete energy dissipation law for the scheme.

\begin{theorem}[Discrete energy stability]
  \label{thm:energy-stability}
  Let \(\{(\psi_h^n,\bA_h^n)\}_{n\ge0}\) be the solution of the discrete system
  \eqref{eq:discrete-tdgl-temporal-psi}-\eqref{eq:discrete-tdgl-temporal-A}. Under the assumptions of Theorem~\ref{thm:discrete-mbp}, and with the additional
  time-step restriction
  \begin{equation}
    \tau\le \frac{2}{\Lambda^{2}},
    \label{eq:energy-timestep-edge}
  \end{equation}
  where $\Lambda$ is defined in \eqref{eq:def-of-lambda}, there holds
  \begin{equation}
    \mathcal{G}_h^{n+1}(\psi_h^{n+1},\bA_h^{n+1})
    \le
    \mathcal{G}_h^n(\psi_h^n,\bA_h^n),
    \qquad \text{for any } n\ge0.
    \label{eq:energy-stability}
  \end{equation}
\end{theorem}

\begin{proof}
  Substituting $\varphi_{h}=\psi_h^{n+1}-\psi_h^n$ and $\bv_{h}=\bA_h^{n+1}-\bA_h^n$ into
  \eqref{eq:discrete-tdgl-temporal-psi}-\eqref{eq:discrete-tdgl-temporal-A}, we obtain
  \begin{align}
    & 0 = \frac{\eta}{\tau}\|\psi_h^{n+1}-\psi_h^n\|_{0,h}^2
      +\Reop\left[a_h^n(\psi_h^{n+1},\psi_h^{n+1}-\psi_h^n)\right]+\Reop\left[\bigl((|\psi_h^n|^2-1)\psi_h^n,\psi_h^{n+1}-\psi_h^n\bigr)_{0,h}\right]
      \nn\\
    &+\frac{1}{\tau}\|\bA_h^{n+1}-\bA_h^n\|_{L^2}^2 + \left(\Curl\bA_h^{n+1}-H,\Curl(\bA_h^{n+1}-\bA_h^n)\right) + \langle \bJ_{h}^{n+1/2},\bA_h^{n+1}-\bA_h^n\rangle_{\mathcal E_h}.
      \label{eq:energy-tested-equations}
  \end{align}
  It is easy to see that
  \begin{align} &\Reop\left(a_h^n(\psi_h^{n+1},\psi_h^{n+1}-\psi_h^n)\right)
                  \notag\\
                &= \frac12a_h^n(\psi_h^{n+1},\psi_h^{n+1}) -\frac12a_h^n(\psi_h^n,\psi_h^n) +\frac12a_h^n(\psi_h^{n+1}-\psi_h^n,\psi_h^{n+1}-\psi_h^n)
                  \label{eq:energy-covariant-polarization}
  \end{align}
  and
  \begin{align}
    &(\Curl\bA_h^{n+1}-H,\Curl(\bA_h^{n+1}-\bA_h^n))_{L^2}
      \notag\\
    &= \frac12\|\Curl\bA_h^{n+1}-H\|_{L^2}^2
      -\frac12\|\Curl\bA_h^n-H\|_{L^2}^2
      +\frac12\|\Curl(\bA_h^{n+1}-\bA_h^n)\|_{L^2}^2.
      \label{eq:energy-magnetic-polarization}
  \end{align}
  Substituting \eqref{eq:energy-covariant-polarization} and \eqref{eq:energy-magnetic-polarization}
  into \eqref{eq:energy-tested-equations} and collecting terms according to the
  discrete energy \eqref{eq:local-discrete-energy}, we obtain
  \begin{align}
    &\mathcal{G}_h^{n+1}-\mathcal{G}_h^n
      +\frac{\eta}{\tau}\|\psi_h^{n+1}-\psi_h^n\|_{0,h}^2
      +\frac12 a_h^n(\psi_h^{n+1}-\psi_h^n,\psi_h^{n+1}-\psi_h^n)
      +\frac1\tau\|\bA_h^{n+1}-\bA_h^n\|_{L^2}^2
      \notag\\
    &\qquad 
      +\frac12\|\Curl(\bA_h^{n+1}-\bA_h^n)\|_{L^2}^2
      \notag\\
    &=\frac14 \sum_{\bm{x}_i \in \mathcal{N}_h} m_i\left[(|\psi_i^{n+1}|^2-1)^2-(|\psi_i^n|^2-1)^2\right] - \Reop\left[\left((|\psi_h^n|^2-1)\psi_h^n,\psi_h^{n+1}-\psi_h^n\right)_{0,h}\right]
    \nn \\
    & \quad +\frac{1}{2}\bigl( a^{n+1}_h(\psi_{h}^{n+1},\psi_{h}^{n+1})
      -a^{n}_h(\psi_{h}^{n+1},\psi_{h}^{n+1})\bigr)-\langle \bJ_{h}^{n+1/2},\bA_{h}^{n+1}-\bA_{h}^{n}\rangle_{\mathcal E_h}\,.
    \nn \\
    & := I_1 + I_2 + I_3 + I_4.      \label{eq:energy-exact-identity}
  \end{align}

  For the first term $I_1$, expanding the quartic difference yields
  \begin{align*}
    \frac14\left[(|\psi_i^{n+1}|^2-1)^2-(|\psi_i^n|^2-1)^2\right]
    &= \frac12(|\psi_i^n|^2-1)\left(|\psi_i^{n+1}|^2-|\psi_i^n|^2\right) \\
    &\quad + \frac14\left(|\psi_i^{n+1}|^2-|\psi_i^n|^2\right)^2.
  \end{align*}
  Combining this with the identity
  \begin{equation*}
    |\psi_i^{n+1}|^2 - |\psi_i^n|^2 - 2\Reop\!\left[\psi_i^n(\psi_i^{n+1}-\psi_i^n)^*\right]
    = |\psi_i^{n+1}-\psi_i^n|^2,
  \end{equation*}
  we obtain
  \begin{align}
    &\frac14\left[(|\psi_i^{n+1}|^2-1)^2-(|\psi_i^n|^2-1)^2\right]
      -\Reop\!\left[(|\psi_i^n|^2-1)\psi_i^n(\psi_i^{n+1}-\psi_i^n)^*\right] \notag\\
    &\quad=
      \frac12(|\psi_i^n|^2-1)|\psi_i^{n+1}-\psi_i^n|^2
      +\frac14\left(|\psi_i^{n+1}|^2-|\psi_i^n|^2\right)^2.
      \label{eq:energy-potential-identity}
  \end{align}
  By Theorem~\ref{thm:discrete-mbp}, the first term on the right-hand side of \eqref{eq:energy-potential-identity} is non-positive
  since $|\psi_i^n| \le 1$. Meanwhile, the second term satisfies
  \begin{equation}
    \left||\psi_i^{n+1}|^2-|\psi_i^n|^2\right|
    \le (|\psi_i^{n+1}|+|\psi_i^n|)\left||\psi_i^{n+1}|-|\psi_i^n|\right|
    \le 2|\psi_i^{n+1}-\psi_i^n|.
  \end{equation}
  Hence
  \begin{align}
    &I_1 + I_2 \le \|\psi_h^{n+1}-\psi_h^n\|_{0,h}^2.
      \label{eq:energy-potential-bound}
  \end{align}

  It remains to estimate the edge-related terms in \eqref{eq:energy-exact-identity}. Recalling the
  definitions of the discrete form $a_h^n(\cdot,\cdot)$ and the discrete supercurrent
  $\bJ_h^{n+1/2}$, we first express both terms as sums over the oriented edges
  $e=(i,j)\in\mathcal{E}_h$:
  \begin{equation}
    I_{3}= \frac{1}{2\kappa^2}\sum_{e\in\mathcal{E}_h}\mu_e \left[\left|\psi_j^{n+1}-\exp(\ri\kappa A_e^{n+1})\psi_i^{n+1}\right|^2 -\left|\psi_j^{n+1}-\exp(\ri\kappa A_e^n)\psi_i^{n+1}\right|^2\right],\label{eq:bilinear-express-edge}
  \end{equation}
  and
  \begin{equation}
     I_{4} = -\sum_{e\in\mathcal{E}_h}\mu_{e} J_e^{n+1/2} (A_e^{n+1}-A_e^n) = -\frac{1}{\kappa}\sum_{e\in\mathcal{E}_h} \mu_e (A_e^{n+1}-A_e^n) \Imop\!\left[(\psi_j^{n+1})^* U^{n}_{e} \psi_i^{n+1}\right].\label{eq:supercurrent-express-edge}
  \end{equation}

  For each fixed edge $e=(i,j)\in\mathcal{E}_h$, expanding the summand in \eqref{eq:bilinear-express-edge}
  gives
  \begin{align}
    &\left|\psi_j^{n+1}-\exp(\ri\kappa A_e^{n+1})\psi_i^{n+1}\right|^2 -\left|\psi_j^{n+1}-\exp(\ri\kappa A_e^n)\psi_i^{n+1}\right|^2\nn \\
    &\quad= -2\Reop\!\left\{ (\psi_j^{n+1})^* \left[\exp(\ri\kappa A_e^{n+1}) - \exp(\ri\kappa A_e^n)\right] \psi_i^{n+1} \right\}\nn \\
    &\quad= -2\Reop\!\left\{ (\psi_j^{n+1})^* U_e^n \psi_i^{n+1} \left[ \exp\!\big(\ri\kappa(A_e^{n+1}-A_e^n)\big) - 1 \right] \right\}.\label{eq:middle-summand-bilinear}
  \end{align}
  Using Euler's formula $\exp(\ri\theta) =\cos\theta + \ri \sin\theta$, the right side of~\eqref{eq:middle-summand-bilinear}
  becomes
  \begin{equation*}
    2\left[1-\cos\big(\kappa(A_e^{n+1}-A_e^n)\big)\right]\Reop\!\left[(\psi_j^{n+1})^* U_e^n \psi_i^{n+1}\right] + 2\sin\big(\kappa(A_e^{n+1}-A_e^n)\big)\Imop\!\left[(\psi_j^{n+1})^* U_e^n \psi_i^{n+1}\right].
  \end{equation*}
  Summing over all edges $e\in\mathcal{E}_h$, we obtain
  \begin{align}
    \quad I_{3}+I_{4}
    &= \frac{1}{\kappa^2}\sum_{e\in\mathcal{E}_h}\mu_e\biggl\{\left[1-\cos\big(\kappa(A_e^{n+1}-A_e^n)\big)\right]\Reop\!\left[(\psi_j^{n+1})^* U_e^n \psi_i^{n+1}\right] \nn\\
    &\mspace{56mu}  +  \left[\sin(\kappa(A_e^{n+1}-A_e^n)) - \kappa(A_e^{n+1}-A_e^n)\right]\Imop\!\left[(\psi_j^{n+1})^* U_e^n \psi_i^{n+1}\right]\biggr\}\nn\\
    &=\frac{1}{\kappa^2}\sum_{e\in\mathcal{E}_h}\mu_e \Reop\biggl\{\left( 1-\ri \kappa (A^{n+1}_{e}-A^{n}_{e})-\exp\left(-\ri\kappa (A^{n+1}_{e}-A^{n}_{e})\right) \right)\left[(\psi_j^{n+1})^* U_e^n \psi_i^{n+1}\right]^{*}\biggr\}\nn\\
    &\leq\frac{1}{2}\sum_{e\in\mathcal{E}_h}\mu_e |A^{n+1}_{e}-A^{n}_{e}|^{2}\left|(\psi_j^{n+1})^* U_e^n \psi_i^{n+1}\right| \,.\label{eq:energy-edge-cancellation}
  \end{align}
 where we have used the Cauchy--Schwarz inequality and the inequality
\(|1-\exp(-\mathrm{i}x)-\mathrm{i}x| \leq x^{2}/2\).
By Theorem~\ref{thm:discrete-mbp}, it follows that
\begin{equation}
    \left|U^{n}_{e}(\psi_j^{n+1})^*\psi_i^{n+1}\right|
    =|\psi_j^{n+1}||\psi_i^{n+1}|\le1\,.\label{eq:ineq-mid-psi}
  \end{equation}
  Combining \eqref{eq:energy-edge-cancellation}-\eqref{eq:ineq-mid-psi} gives
  \begin{equation}
   I_{3} + I_{4} \le\frac12\sum_{e\in\mathcal{E}_{h}}\mu_e|A_e^{n+1}-A_e^n|^2\le \frac{\Lam^{2}}{2} \|\bA_h^{n+1}-\bA_h^n\|_{L^2(\Omega)}^2\,.
      \label{eq:energy-link-bound}
  \end{equation}
  where the last step follows from \eqref{eq:upper-bound-edge-norms}-\eqref{eq:def-of-lambda}.

  Substituting \eqref{eq:energy-potential-bound} and \eqref{eq:energy-link-bound}
  into \eqref{eq:energy-exact-identity} yields the dissipation inequality
  \begin{align}
    &\mathcal{G}_h^{n+1}-\mathcal{G}_h^n
      +\left(\frac{\eta}{\tau}-1\right)\|\psi_h^{n+1}-\psi_h^n\|_{0,h}^2
      +\frac12 a_h^n(\psi_h^{n+1}-\psi_h^n,\psi_h^{n+1}-\psi_h^n)
      \nn\\
    &\quad+\big(\frac1\tau-\frac{\Lambda^{2}}{2}\big)
      \|\bA_h^{n+1}-\bA_h^n\|_{L^2(\Omega)}^2 + \frac12\|\Curl(\bA_h^{n+1}-\bA_h^n)\|_{L^2(\Omega)}^2\le0.
      \label{eq:energy-dissipation}
  \end{align}
  By noting that $\tau \le \frac{\eta}{2}$ and \eqref{eq:energy-timestep-edge}, there hold
  \begin{equation*}
    \frac{\eta}{\tau}-1 \ge 1 \quad\text{and}\quad \frac{1}{\tau} -\frac{\Lambda^{2}}{2}\ge 0\,.
  \end{equation*}
  Moreover $\mu_e\ge0$ gives
  \begin{equation*}
    a_h^n(\psi_h^{n+1}-\psi_h^n,\psi_h^{n+1}-\psi_h^n)\ge0,
  \end{equation*}
  Consequently, we obtain the desired energy dissipation property
  \begin{equation*}
    \mathcal{G}_h^{n+1}(\psi_h^{n+1},\bA_h^{n+1})
    -\mathcal{G}_h^n(\psi_h^n,\bA_h^n) \le 0.
  \end{equation*}
\end{proof}

\begin{remark}
  Lemma~\ref{lem:local-discrete-gauge-covariance} establishes the discrete gauge-invariant property.
  It follows that any solution in an arbitrary gauge can be transformed into a temporal-gauge
  solution by a discrete gauge transformation preserving both $|\psi_i^n|$ and $\mathcal{G}_h^n$.
  Consequently, the conclusions of Theorems~\ref{thm:discrete-mbp} and~\ref{thm:energy-stability} are
  not restricted to the temporal gauge, and both the discrete maximum bound
  principle and energy stability hold under any admissible gauge.
\end{remark}

\section{Numerical experiments}
\label{sec:numerical-experiments}

In this section, we present several numerical examples to demonstrate the
accuracy and the structure-preserving property of the proposed scheme. All
computations are carried out with the finite element library MFEM~\cite{Anderson2021}, and the
meshes are generated by GMSH~\cite{Geuzaine2009a}. All meshes used in the numerical experiments
satisfy the nonnegative edge-weight condition~\eqref{eq:mesh-degree-assume}. For simplicity, we take $\eta=1$ in
all the examples below.

\subsection{Convergence test on a manufactured solution}
\label{subsec:mms-convergence}

\begin{example}[An artificial example]
  \label{ex:mms-convergence}
  We consider an artificial problem on $\Omega=[0,1]^{2}$, defined by
  \begin{empheq}[left=\empheqlbrace]{align}
    &\partial_t\psi+\left(\dfrac{\ri}{\kappa}\nabla+\bA\right)^2\psi
    +\bigl(|\psi|^2-1\bigr)\psi=f,\label{eq:mms-sources-psi}\\
    &\partial_t\bA+\Curl\!\bigl(\Curl\bA\bigr)+\Reop\!\left[\psi^*\left(\dfrac{\ri}{\kappa}\nabla+\bA\right)\psi\right] = \Curl H + \mathbf{g},\label{eq:mms-sources-A}
  \end{empheq}
  with the boundary and the initial conditions
  \begin{align*}
    \left(\dfrac{\ri}{\kappa}\nabla+\bA\right)\psi\cdot\bn=0,
    \qquad
    &(\Curl\bA) \times \bn=H \times \bn,
      \qquad
      \bA\cdot\bn=0
      \quad\text{on }\partial\Omega,\\
    &\psi(\cdot,0)=\psi_0,
      \qquad
      \bA(\cdot,0)=\bA_0 \quad \text{at~}t=0\,.
  \end{align*}
  where we take $\kappa=1$. The functions $f$, $\mathbf g$, $\psi_{0}$ and $\bA_{0}$ are
  chosen correspondingly to the smooth exact solution
  \begin{equation*}
    \psi(x,y,t)=e^{-t}\bigl(\cos(2\pi x)+\ri\cos(\pi y)\bigr) ,\quad\bA(x,y,t)=\begin{bmatrix}
      e^{t-y}\sin(\pi x)\\
      e^{t-x}\sin(2\pi y)
    \end{bmatrix}
  \end{equation*}
  and
  \begin{equation*}
    H(x,y,t):=\Curl\bA=-e^{t-x}\sin(2\pi y)+e^{t-y}\sin(\pi x).
  \end{equation*}
  We set the terminal time $T=1$ in this example.

  We employ uniform meshes with mesh size \(h\) and solve the discrete system
  using the proposed scheme. To test the convergence, we set the time
  step \(\tau=h^2\) and take \(h=\sqrt{2}/8,\sqrt{2}/16,\sqrt{2}/32,\sqrt{2}/64\).
  Table~\ref{tab:mms-convergence} reports the $L^{2}$-norm errors of \(\psi_h\), \(\bA_h\), and
  \(\Curl\bA_h\), together with the convergence rates. The numerical results
  indicate that the errors in \(\bA_h\) and \(\Curl\bA_h\) are of order
  \(O(h)\), whereas the error in \(\psi_h\) is of order \(O(h^2)\).
  \begin{table}[htb]
    \centering
    \caption{Numerical errors and convergence rates for Example~\ref{ex:mms-convergence}.}\label{tab:mms-convergence}
    \begin{tabular}{c|cc|cc|cc}
      \toprule
      $h$
      & $\|\psi_h^n-\psi\|_{L^2(\Omega)}$ & Rate
      & $\|\bA_h^n-\bA\|_{L^2(\Omega)}$ & Rate
      & $\|\Curl\bA_h^n-\Curl\bA\|_{L^2(\Omega)}$ & Rate\\
      \midrule
      $\sqrt{2}/8$  & $5.885\times10^{-2}$ & ---  & $3.340\times10^{-1}$ & ---  & $2.664\times10^{-1}$ & ---\\
      $\sqrt{2}/16$ & $1.447\times10^{-2}$ & $2.02$ & $1.642\times10^{-1}$ & $1.02$ & $1.338\times10^{-1}$ & $0.99$\\
      $\sqrt{2}/32$ & $3.599\times10^{-3}$ & $2.01$ & $8.178\times10^{-2}$ & $1.01$ & $6.697\times10^{-2}$ & $1.00$\\
      $\sqrt{2}/64$ & $8.981\times10^{-4}$ & $2.00$ & $4.086\times10^{-2}$ & $1.00$ & $3.349\times10^{-2}$ & $1.00$\\
      \bottomrule
    \end{tabular}
  \end{table}
\end{example}

\subsection{Numerical verification of energy dissipation and discrete maximum bound}
\label{subsec:energy-mbp}

\begin{example}[An L-shaped domain]
  \label{ex:lshape-unit}
  We consider the temporal-gauge TDGL equations on the L-shaped domain
  \[
    \Omega=(-0.5,0.5)^2\setminus\bigl([0,0.5]\times[-0.5,0]\bigr),
  \]
  whose re-entrant corner is located at the origin. The discrete system is
  solved by the proposed scheme \eqref{eq:discrete-tdgl-temporal-psi}--\eqref{eq:discrete-tdgl-temporal-A} on uniform meshes. The mesh is
  generated by splitting every retained unit square of an \(M\times M\) grid with
  \(M=32\). The physical parameters are
  \[
    \kappa=10,\quad \eta=1,\quad H=5,\quad \psi_h^0=0.6+0.8\,\ri, \quad \bA_h^0=\boldsymbol0,\quad \tau=0.1,\quad T=50.
  \]
  The \(L^\infty\)-norm of $\psi^{n}_{h}$ and the discrete energy \eqref{eq:local-discrete-energy}
  are recorded at all \(501\) time levels. Representative values are listed in
  Table~\ref{tab:lshape-energy-mbp}. The discrete energy satisfies the dissipation property at each time
  step, and the maximum bound of $\psi_{h}^{n}$ remains bounded by one for all
  \(n\). This is consistent with the discrete maximum bound in Theorem~\ref{thm:discrete-mbp} and
  the energy stability in Theorem~\ref{thm:energy-stability}. The two diagnostics are plotted over all
  time steps in Figure~\ref{fig:lshape-diagnostics}, and snapshots of \(|\psi_h|\) and \(\Curl\bA_h\) are
  given in Figures~\ref{fig:lshape-psi} and~\ref{fig:lshape-curl}, respectively. The vortex pattern agrees with
  previous computations in which equation (\ref{eq:tdgl-zero-phi:A}) for $\bA$ was directly
  approximated by N\'ed\'elec elements (see \cite{Duan2022,Hong2023a,Ma2023a}). Hence, the use of N\'ed\'elec
  elements for the temporal gauge TDGL is widely considered ``{\it safe}''.

  \begin{table}[htbp]
    \centering \small \setlength{\tabcolsep}{5pt}
    \caption{Discrete energy \eqref{eq:local-discrete-energy} and $L^\infty$-norm of the order parameter on the
      unit L-shaped domain with $\kappa=10$, $H=5$, and $\tau=0.1$.}
    \label{tab:lshape-energy-mbp}
    \begin{tabular}{l|ccccccc}
      \toprule
      $t$ & $0$ & $1$ & $2$ & $5$ & $10$ & $20$ & $50$\\
      \midrule
      $\mathcal G_h^n$ & $9.3750000$ & $0.1177322$ & $0.1136975$ & $0.1092171$ & $0.1022367$ & $0.0993610$ & $0.0992546$\\
      $\|\psi_h^n\|_{L^\infty(\Omega)}$ & $1.0000000$ & $0.9507678$ & $0.9165344$ & $0.9229063$ & $0.9290705$ & $0.9282978$ & $0.9275137$\\
      \bottomrule
    \end{tabular}
  \end{table}

\begin{figure}[htbp]
  \centering
  \begin{minipage}[t]{0.35\linewidth}
    \centering
    \includegraphics[width=\linewidth]{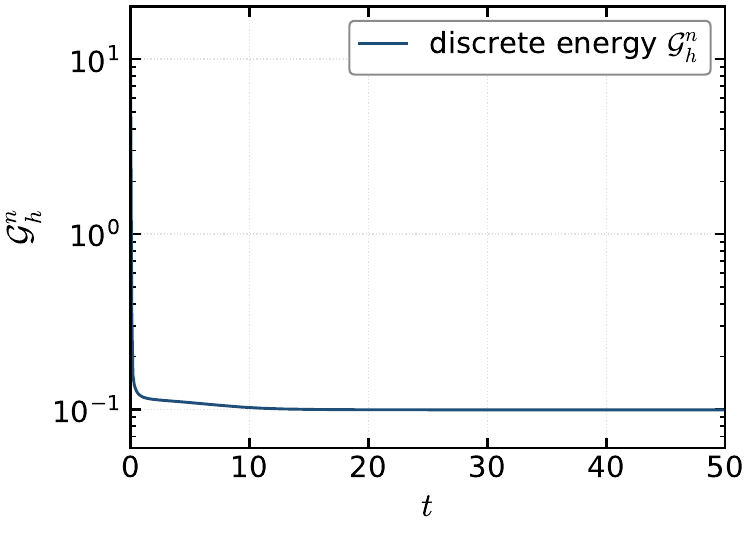}
  \end{minipage}\qquad
  \begin{minipage}[t]{0.35\linewidth}
    \centering
    \includegraphics[width=\linewidth]{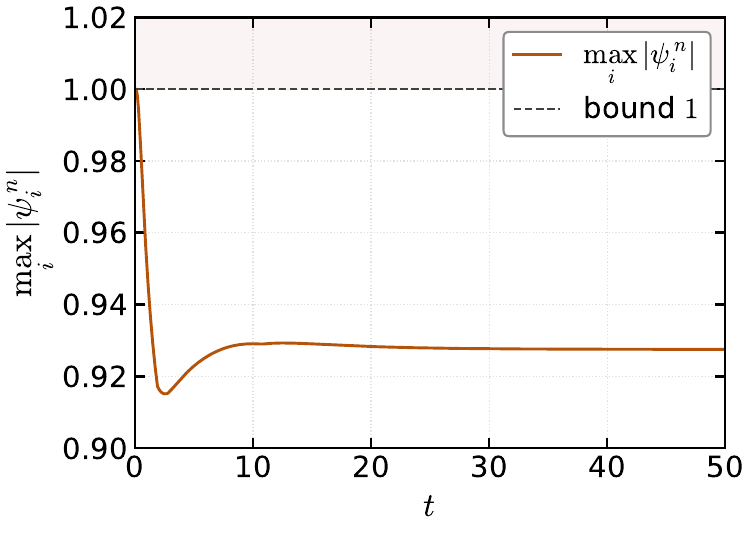}
  \end{minipage}
  \caption{Unit L-shaped domain ($\kappa=10$, $H=5$, $\tau=0.1$). Left: discrete energy
    $\mathcal G_h^n$ over $501$ time levels. Right: $\|\psi_h^n\|_{L^\infty(\Omega)}$ at the
    same time levels.}
  \label{fig:lshape-diagnostics}
\end{figure}

\begin{figure}[htbp]
  \centering
  \includegraphics[width=0.98\linewidth]{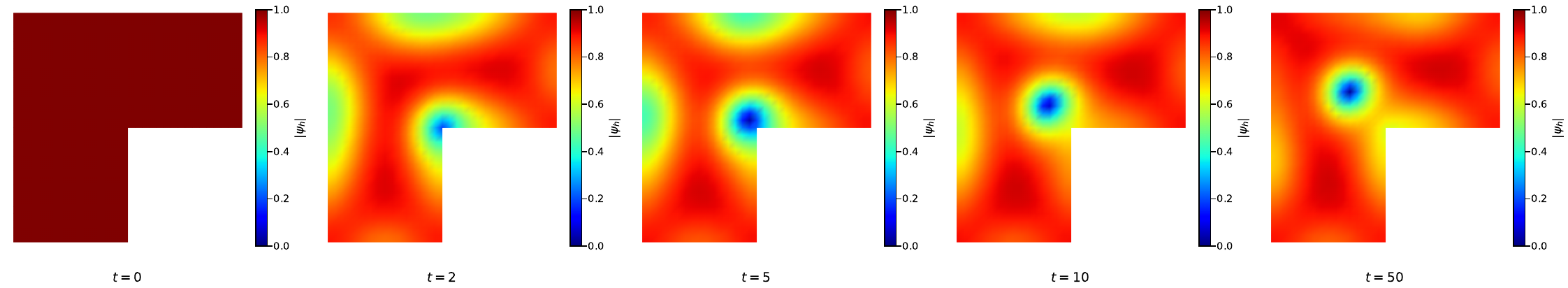}
  \caption{Snapshots of $|\psi_h|$ on the unit L-shaped domain with $\kappa=10$, $H=5$,
    and $\tau=0.1$, at $t=0,2,5,10,50$.}
  \label{fig:lshape-psi}
\end{figure}
\begin{figure}[htbp]
  \centering
  \includegraphics[width=0.98\linewidth]{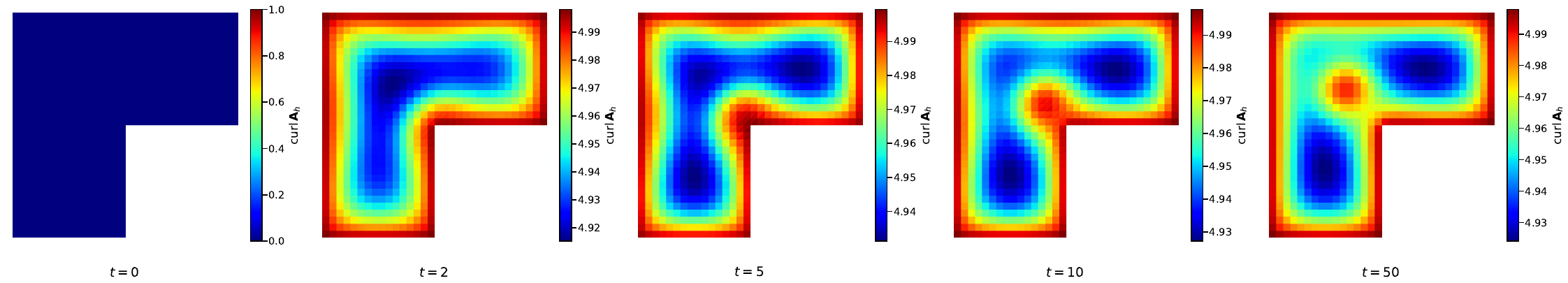}
  \caption{Snapshots of $\Curl\bA_h$ on the unit L-shaped domain with $\kappa=10$,
    $H=5$, and $\tau=0.1$, at $t=0,2,5,10,50$.}
  \label{fig:lshape-curl}
\end{figure}

\end{example}

\begin{example}[Four vortex lines in a unit cube]
  \label{ex:cube-3d}
  As a three-dimensional test we solve the temporal-gauge TDGL equations on the
  unit cube $\Omega=[0,1]^3$ with physical parameters
  \[
    \kappa=10,\quad \eta=1,\quad H=(0,0,3.5),\quad \psi_h^0=0.8+0.6\,\ri,\quad \bA_h^0=\boldsymbol0,\quad \tau=0.1,\quad T=100.
  \]
  The computational domain is discretized by a uniform tetrahedral mesh with
  $9261$ vertices and $48000$ tetrahedra.

  Figure~\ref{fig:cube-diagnostics} shows that, on the tetrahedral mesh, the discrete energy decays
  monotonically and the pointwise bound $\|\psi_h^n\|_{L^\infty(\Omega)}\le 1$ is satisfied at
  every time step; thus the theoretical guarantees of Theorems~\ref{thm:discrete-mbp} and \ref{thm:energy-stability} carry
  over to three dimensions. Figure~\ref{fig:cube-vortices} displays the $|\psi_h|=0.3$ isosurface at
  $t=100$, revealing four vortex lines aligned with the applied field.

\begin{figure}[htbp]
  \centering
  \begin{minipage}[t]{0.35\linewidth}
    \centering
    \includegraphics[width=\linewidth]{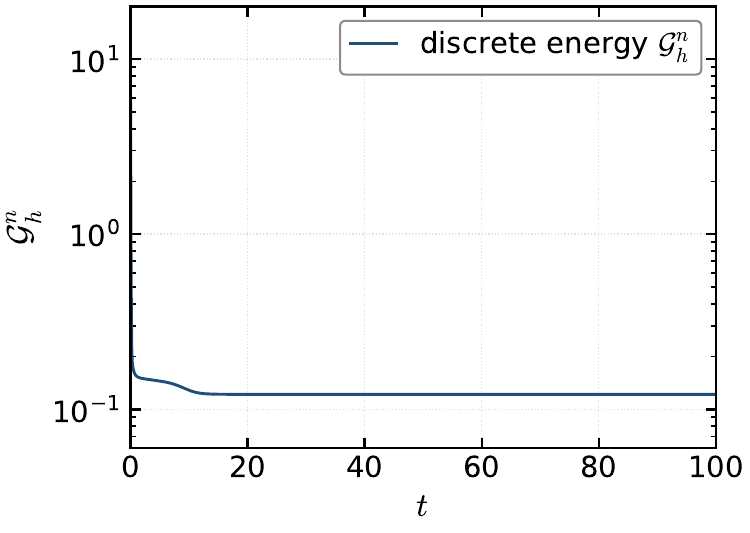}\\[-2pt]
  \end{minipage}\qquad
  \begin{minipage}[t]{0.35\linewidth}
    \centering
    \includegraphics[width=\linewidth]{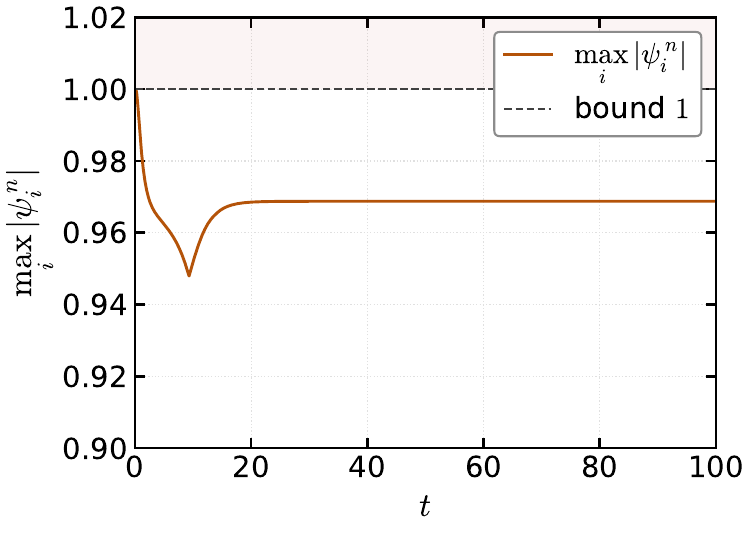}\\[-2pt]
  \end{minipage}
  \caption{Unit cube ($\kappa=10$, $H=(0,0,3.5)$, $\tau=0.1$). Left: discrete energy
    $\mathcal G_h^n$ over $1001$ time levels. Right: $\|\psi_h^n\|_{L^\infty(\Omega)}$ at the
    same time levels.} \label{fig:cube-diagnostics}
\end{figure}

\begin{figure}[htbp]
  \centering
  \begin{minipage}[t]{0.35\linewidth}
    \centering
    \includegraphics[width=\linewidth]{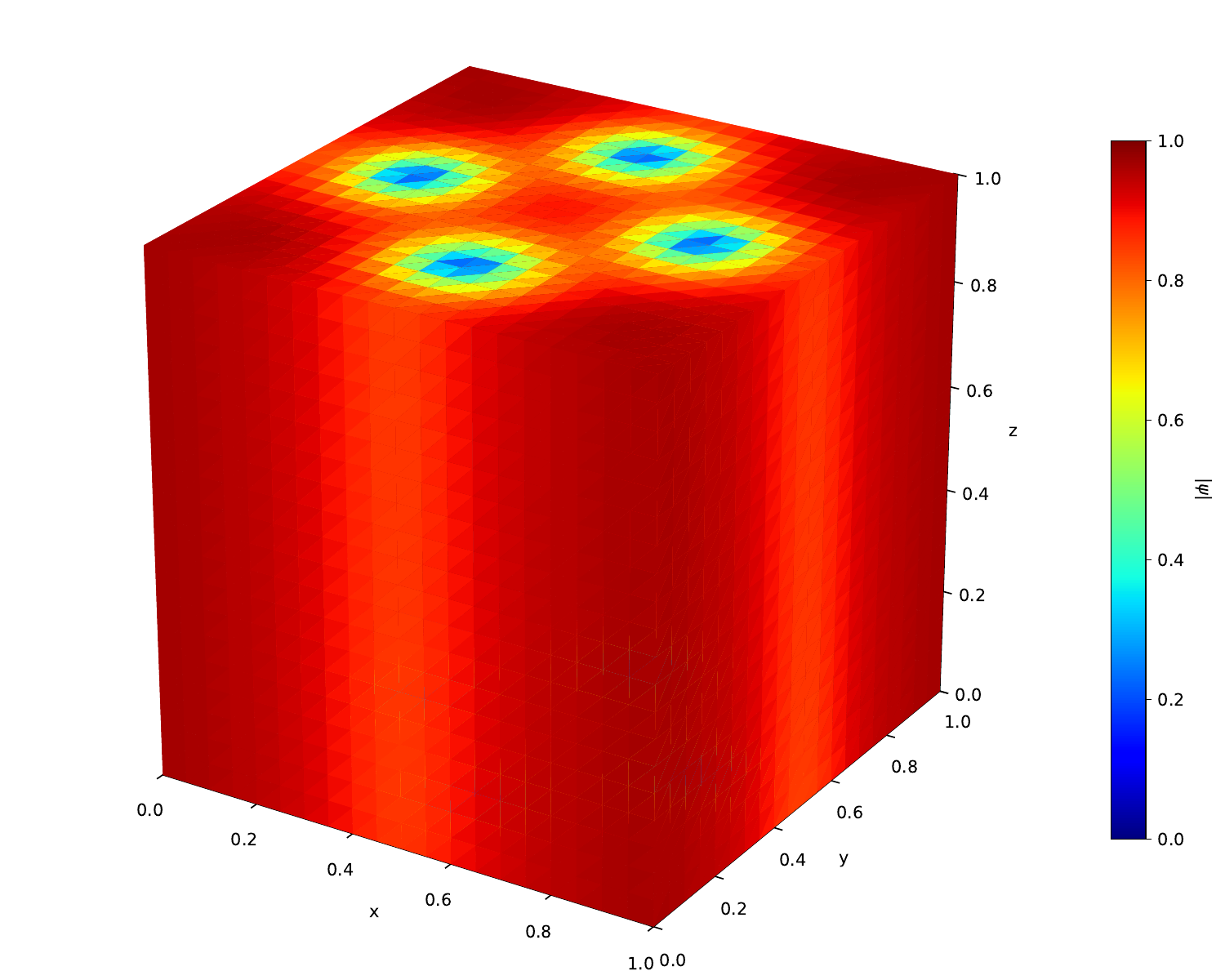}
  \end{minipage}\qquad
  \begin{minipage}[t]{0.35\linewidth}
    \centering
    \includegraphics[width=\linewidth]{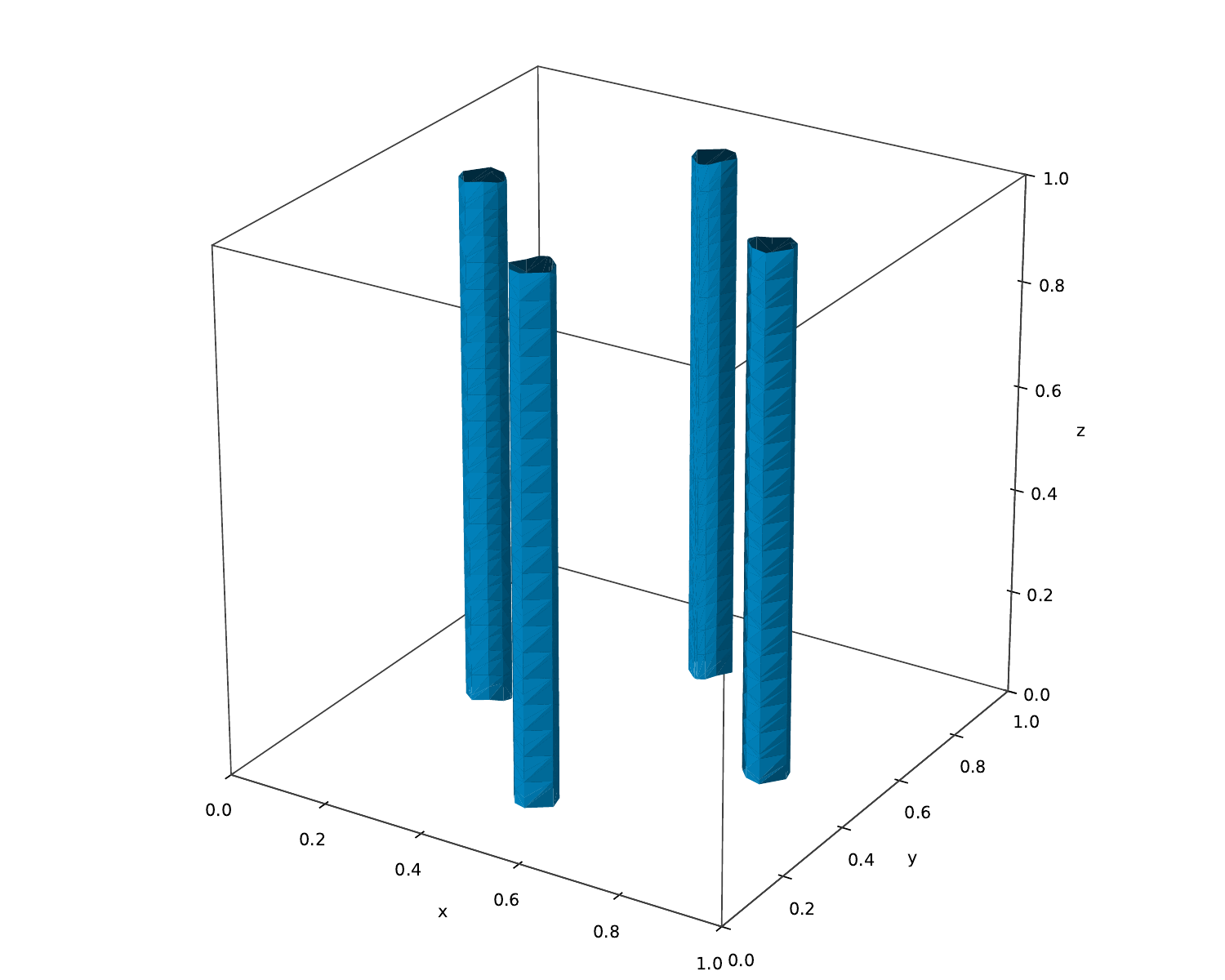}
  \end{minipage}
  \caption{Three-dimensional vortex structure at $t=100$. The left plot shows
    the boundary distribution of $|\psi_h|$, while the right plot shows the
    $|\psi_h|=0.3$ isosurface with four vortex lines aligned with the $z$-axis.}
  \label{fig:cube-vortices}
\end{figure}
\end{example}

\subsection{Vortex dynamics on larger domains with re-entrant corners}
\label{subsec:long-time-evolution}

In this section, we study the vortex dynamics on larger planar domains with
re-entrant corners, where an increased number of vortices is present. Two
different geometries are considered: a large L-shaped domain and a circular disk
with a triangular boundary defect. To assess the influence of the
gauge-invariant discretization, each computation is carried out with both the
proposed scheme and a conventional temporal-gauge discretization that lacks the
gauge invariance. The conventional discretization reads as follows. Starting
from $(\psi_h^n,\bA_h^n)\in V_h\times X_h$, we determine $(\psi_h^{n+1},\bA_h^{n+1})$
successively from
\begin{align*}
  &\eta\Big(\frac{\psi_h^{n+1}-\psi_h^n}{\tau},\varphi_h\Big)
    +\left(\frac{\ri}{\kappa}\nabla\psi_h^{n+1}+\bA_h^n\psi_h^{n+1},
    \frac{\ri}{\kappa}\nabla\varphi_h+\bA_h^n\varphi_h\right)
    +\bigl((|\psi_h^n|^2-1)\psi_h^n,\varphi_h\bigr)=0 \, ,
    \nn \\
  & \Big(\frac{\bA_h^{n+1}-\bA_h^n}{\tau},\bv_h\Big)
    +\bigl(\Curl\bA_h^{n+1},\Curl\bv_h\bigr)
    +\left(\Reop\!\left[(\psi_h^n)^{*}
    \left(\frac{\ri}{\kappa}\nabla\psi_h^n+\bA_h^n\psi_h^n\right)\right],\bv_h\right)
    =\bigl(H,\Curl\bv_h\bigr)
\end{align*}
for all test functions $\varphi_h\in V_h$ and $\bv_h\in X_h$. The same mesh and time step
are used in both computations.

\begin{example}[A large L-shaped domain]
  \label{ex:lshape10}
  In the fourth example, we simulate the vortex dynamics in a large L-shaped
  domain
  \[
    \Omega=(-10,10)^2\setminus\bigl([0,10]\times[-10,0]\bigr),
  \]
  whose re-entrant corner again lies at the origin. Here we use a symmetric mesh
  with $h=0.125$, which consists of $22731$ vertices, $44820$ triangles and
  $67550$ edges. The initial conditions are taken to be the perfectly
  superconducting state with $\psi_h^0=0.8+0.6\,\ri$ and $\bA_h^0=\boldsymbol0$,
  while the remaining physical parameters are $\kappa=2$, $\eta=1$, $H=0.8$, $\tau=0.1$ and
  $T=2000$.

  The energy decreases monotonically and the $L^\infty$-norm satisfies
  $\|\psi_h^n\|_{L^\infty(\Omega)}\le1$ at every time level, in agreement with Theorem~\ref{thm:discrete-mbp}
  and~\ref{thm:energy-stability}. For comparison, under the same mesh and time step, the conventional
  method produces a nonphysical giant vortex at the re-entrant corner, and this
  vortex grows progressively over time. The proposed gauge-invariant FEM does
  not exhibit this artifact. Figures~\ref{fig:lshape10-psi} and~\ref{fig:lshape10-std-psi} show the evolution of
  \(|\psi_h^n|\) for the two discretizations.

  We further examine the conventional scheme in two additional cases: (i) with
  the higher-order $(\mathcal P_2,\mathcal{ND}_2)$ pair, and (ii) under both
  mesh and time-step refinement to $h=0.0625$ and $\tau=0.05$. The physical
  parameters and initial conditions are unchanged. Figures~\ref{fig:lshape10-std-p2nd2} and~\ref{fig:lshape10-std-refined} show that
  although these modifications delay the formation of the nonphysical giant
  vortex at the re-entrant corner, it still appears. We also note that adding a
  term $-\epsilon \nabla(\mathrm{div} \bA)$ to equation (\ref{eq:tdgl-zero-phi:A}) and then applying a mixed
  finite element method does not eliminate this artifact, see \cite{Tain2024a}.

\begin{figure}[tbp]
  \centering
  \includegraphics[width=0.98\linewidth]{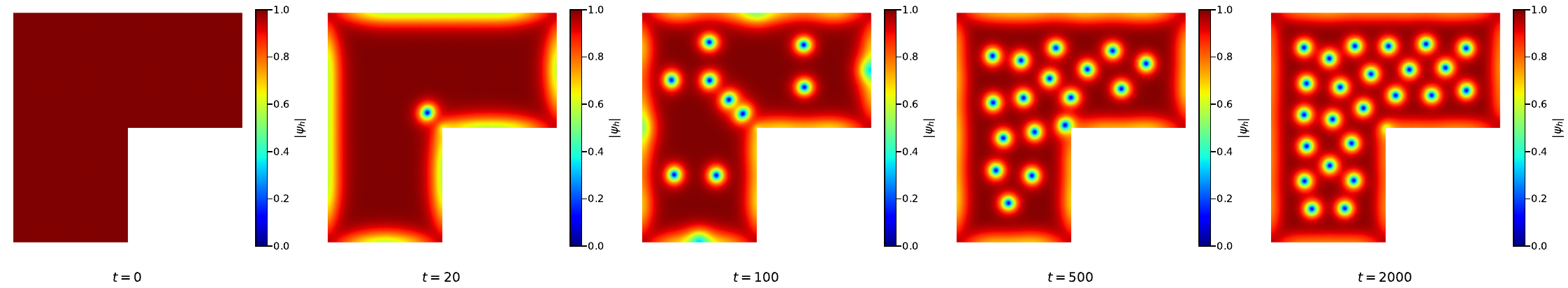}
  \caption{Evolution of $|\psi_h|$ on the large L-shaped domain for the
    gauge-invariant discretization with $\kappa=2$, $H=0.8$, and $\tau=0.1$, at
    $t=0,20,100,500,2000$.}
  \label{fig:lshape10-psi}
\end{figure}

\begin{figure}[tbp]
  \centering
  \includegraphics[width=0.98\linewidth]{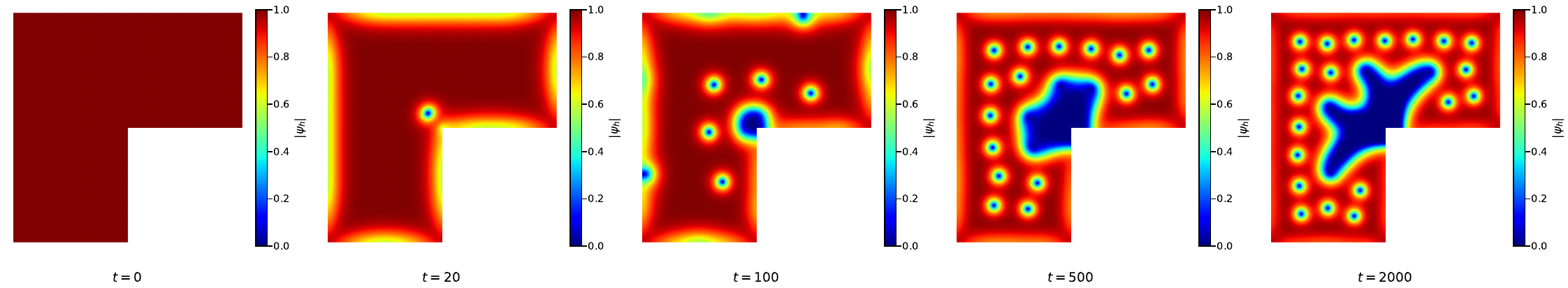}
  \caption{Evolution of $|\psi_h|$ on the large L-shaped domain for the conventional
    temporal-gauge discretization with $\kappa=2$, $H=0.8$, and $\tau=0.1$, at
    $t=0,20,100,500,2000$.}
  \label{fig:lshape10-std-psi}
\end{figure}

\begin{figure}[tbp]
  \centering
  \includegraphics[width=0.98\linewidth]{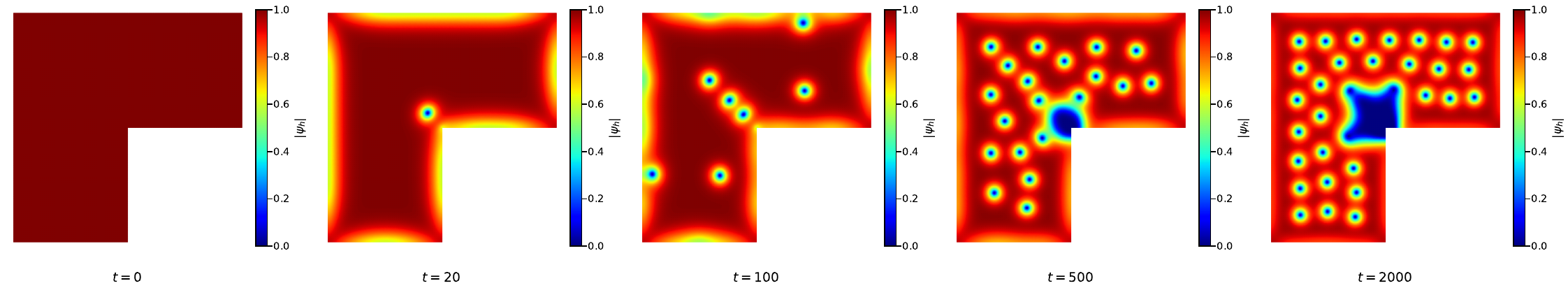}
  \caption{Evolution of $|\psi_h|$ on the large L-shaped domain for the conventional
    temporal-gauge discretization with $\kappa=2$ and $H=0.8$, using the higher-order
    $(\mathcal P_2,\mathcal{ND}_2)$ pair with $h=0.125$ and $\tau=0.1$, at
    $t=0,20,100,500,2000$.}
  \label{fig:lshape10-std-p2nd2}
\end{figure}

\begin{figure}[tbp]
  \centering
  \includegraphics[width=0.98\linewidth]{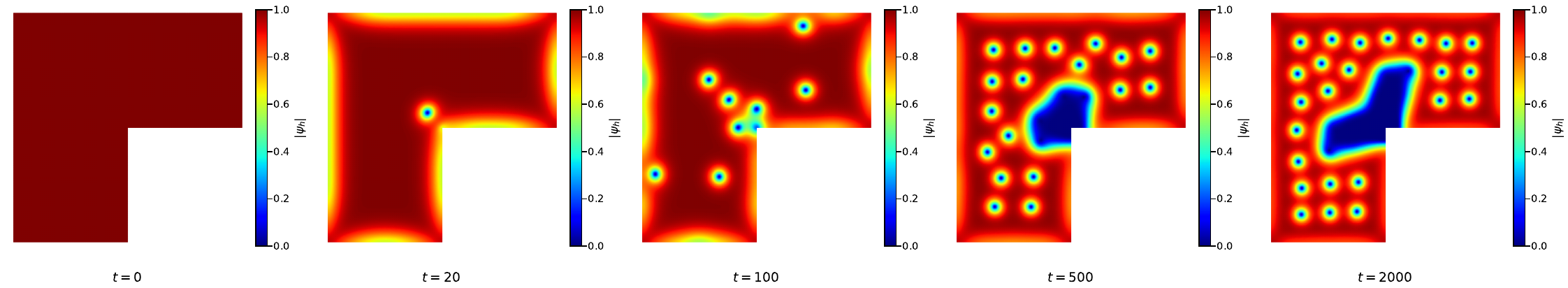}
  \caption{Evolution of $|\psi_h|$ on the large L-shaped domain for the conventional
    temporal-gauge discretization with $\kappa=2$ and $H=0.8$, under both mesh and
    time-step refinement to $h=0.0625$ and $\tau=0.05$, at $t=0,20,100,500,2000$.}
  \label{fig:lshape10-std-refined}
\end{figure}
\end{example}

\begin{example}[A disk with a re-entrant corner]
  \label{ex:circle}
  In the last example, we simulate the vortex dynamics in a circular disk with a
  triangular defect on the boundary, i.e.,
  \[
    \Omega=\bigl\{x^2+y^2<25\bigr\}\setminus\bigl\{|x+5|+|y|<1\bigr\},
  \]
  whose boundary has a re-entrant corner at $(-4,0)$.
  Alstr{\o}m et al.~\cite{alstrom} first computed this example using COMSOL with
  conventional Lagrange FEMs and observed spurious giant vortices numerically.
  Here we use a symmetric
  mesh, which has $23415$ vertices, $46312$ triangles and $69726$ edges. The
  physical parameters are set to $\kappa=4$, $\eta=1$, and $H=0.8$. The initial
  conditions are taken to be the perfectly superconducting state, with
  $\psi_h^0=0.8+0.6\,\ri$ and $\bA_h^0=\boldsymbol{0}$. The time step is $\tau=0.1$,
  and the computation is advanced to $T=4000$, i.e., a total of $40000$ time
  steps. The evolution of $|\psi_h|$ computed by the two methods with the same mesh
  and time step is shown in Figures~\ref{fig:circle-psi} and~\ref{fig:circle-std-psi}. In the present computation, the
  numerical solutions generated by both schemes exhibit energy decay and satisfy
  $\|\psi_h^n\|_{L^\infty(\Omega)}\le 1$ at all time levels. By contrast, the conventional
  method with N\'ed\'elec approximations
  produces a spurious large vortex near the re-entrant corner that grows
  progressively in time, whereas the proposed scheme produces no such artifacts.
  \begin{figure}[htbp]
    \centering
    \includegraphics[width=0.98\linewidth]{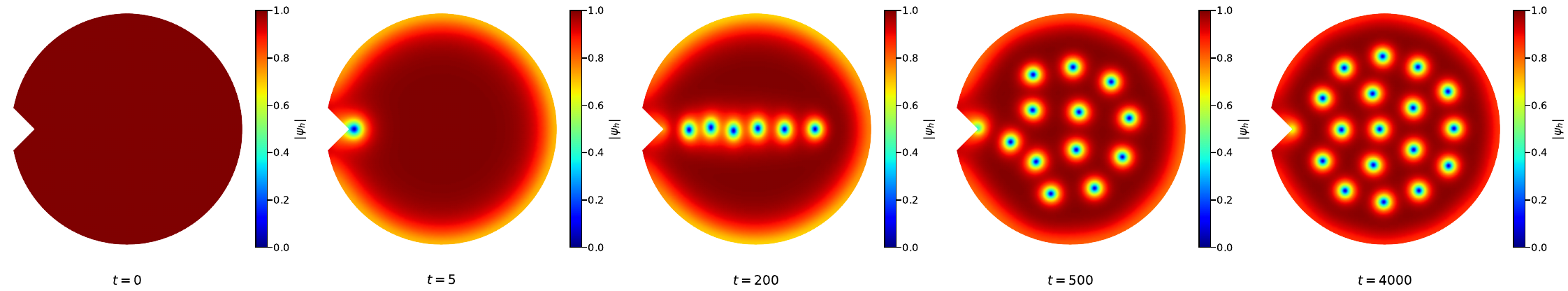}
    \caption{Evolution of $|\psi_h|$ on the circular domain with a re-entrant
      corner for the gauge-invariant discretization with $\kappa=4$, $H=0.8$, and
      $\tau=0.1$, at $t=0,5,200,500,4000$.}
    \label{fig:circle-psi}
  \end{figure}
  \begin{figure}[htpb]
    \centering
    \includegraphics[width=0.98\linewidth]{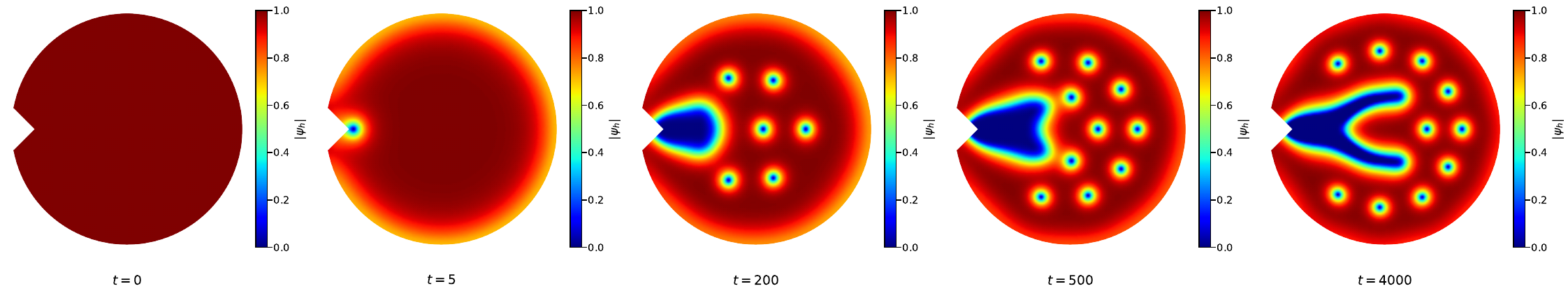}
    \caption{Evolution of $|\psi_h|$ on the circular domain with a re-entrant
      corner for the conventional temporal-gauge discretization with $\kappa=4$,
      $H=0.8$, and $\tau=0.1$, at $t=0,5,200,500,4000$.}
    \label{fig:circle-std-psi}
  \end{figure}
\end{example}

\section{Conclusion}
\label{sec:conclusion}

In this paper, we have developed a lowest-order gauge-invariant finite element
method for the time-dependent Ginzburg--Landau equations. The proposed scheme is linear
and decoupled. It preserves a discrete maximum bound principle for the order
parameter, $|\psi_h^n|\le 1$, and satisfies a discrete energy law under mild
restrictions on the time step and the mesh. The numerical results demonstrate
the effectiveness of the proposed gauge-invariant scheme for long-time vortex
simulations on domains with re-entrant corners. At the tested spatial and
temporal resolutions, the conventional temporal-gauge discretization develops a
nonphysical giant vortex near the corners, whereas the proposed scheme does not
exhibit this artifact. The experiments also support the discrete maximum bound
principle and energy dissipation law and show the expected convergence rates for
a smooth manufactured solution. Several questions remain open for future work,
including the extension of the method to higher-order finite elements,
optimal-order error estimates for solutions with corner singularities, and
adaptive mesh refinement strategies.

\section*{Declarations}

\noindent\textbf{Conflict of interest.}
The authors declare that they have no conflict of interest.

\medskip
\noindent\textbf{Availability of data and material.}
The meshes and the MFEM-based solver used to generate the numerical results of
Section~\ref{sec:numerical-experiments} are provided as supplementary material accompanying this manuscript.

\medskip
\noindent\textbf{Use of AI Tools.}
During the preparation of this manuscript, the authors used OpenAI Codex to assist with programming and with writing plotting code used to generate figures, and to polish portions of the text. The authors assume responsibility for all content.


\end{document}